\documentclass[a4paper,10pt,reqno]{amsart}
\usepackage[foot]{amsaddr}
\usepackage{amsmath}
\usepackage{amsthm,enumerate, esint}
\usepackage{mathtools}
\usepackage{hyperref,url}
\usepackage{graphicx}
\usepackage{amssymb}
\usepackage{hyperref}
\hypersetup{colorlinks=true, citecolor=blue}
\usepackage{appendix}

\usepackage[english]{babel}

\usepackage{fancyhdr}

\usepackage{calc}
\usepackage{url}

\usepackage[text={6in,8.6in},centering]{geometry}

\usepackage{srcltx}

\usepackage[T1]{fontenc}

\usepackage[usenames,dvipsnames]{color}

\theoremstyle{plain}
\newtheorem{theorem}{Theorem}[section]
\theoremstyle{plain}
\newtheorem{lemma}{Lemma}[section]
\newtheorem{proposition}{Proposition}[section]

\newtheorem{definition}{Definition}[section]
\newtheorem{remark}{Remark}[section]

\renewcommand{\[}{\left[}
\renewcommand{\]}{\right]}
\newcommand{\eps}{\varepsilon}

\newcommand{\To}{\longrightarrow}
\newcommand{\be} {\begin{equation}}
\newcommand{\ee} {\end{equation}}
\newcommand{\bea} {\begin{eqnarray}}
\newcommand{\eea} {\end{eqnarray}}
\newcommand{\Bea} {\begin{eqnarray*}}
	\newcommand{\Eea} {\end{eqnarray*}}

\newcommand{\al} {\alpha}
\newcommand{\ba} {\beta}
\newcommand{\de} {\delta}

\newcommand{\ga} {\gamma}
\newcommand{\Ga} {\Gamma}

\newcommand{\om} {\omega}
\newcommand{\De} {\Delta}
\newcommand{\la} {\lambda}

\newcommand{\no} {\nonumber}
\newcommand{\noi} {\noindent}
\newcommand{\lab} {\label}
\newcommand{\va} {\varphi}

\newcommand{\R}{\mathbb R}
\newcommand{\N}{\mathbb N}
\newcommand{\Rn}{\mathbb R^N}

\newcommand{\deb}{\rightharpoonup}
\newcommand{\Hs}{\dot{H}^s(\mathbb{R}^{N})}

\newcommand{\hhms}{(\dot{H}^{s}\times\dot{H}^s)'}
\newcommand{\hms}{(\dot{H}^{s})'}
\newcommand{\phs}{\dot{H}^s\times\dot{H}^s}
\catcode`\@=11

\newcommand{\authorfootnotes}{\renewcommand\thefootnote{\@fnsymbol\c@footnote}}%
\allowdisplaybreaks

\def\N{{I\!\!N}}

\numberwithin{equation}{section} \allowdisplaybreaks

\begin{document}
        \title{Fractional elliptic systems with critical nonlinearities}

\date{}

\author[M. Bhakta]{Mousomi Bhakta\textsuperscript{1}}
\address{\textsuperscript{1}Department of Mathematics, Indian Institute of Science Education and Research, Dr. Homi Bhaba Road, Pune-411008, India}
\email{mousomi@iiserpune.ac.in}

\author[S. Chakraborty]{Souptik Chakraborty\textsuperscript{1}}
\email{souptik.chakraborty@students.iiserpune.ac.in}

\author[O.H. Miyagaki]{Olimpio H. Miyagaki \textsuperscript{2}}
\address{\textsuperscript{2}Departamento de Matem\'atica, Universidade Federal de S\~ao Carlos, S\~ao Carlos-SP, 13565-905, Brazil}
\email{ohmiyagaki@gmail.com}

\author[P. Pucci]{Patrizia Pucci\textsuperscript{3}}
\address{\textsuperscript{3}Dipartimento di Matematica e Informatica, Universit\`a degli Studi di Perugia --
Via Vanvitelli 1, I-06123 Perugia, Italy}
\email{patrizia.pucci@unipg.it}


\begin{abstract}
This paper deals with existence, uniqueness/multiplicity of positive solutions to the following nonlocal system of equations:
\begin{equation}
	\tag{$\mathcal S$}\label{MAT1}
	\left\{\begin{aligned}
		&(-\Delta)^s u = \frac{\alpha}{2_s^*}|u|^{\alpha-2}u|v|^{\beta}+f(x)\;\;\text{in}\;\mathbb{R}^{N},\\
		&(-\Delta)^s v = \frac{\beta}{2_s^*}|v|^{\beta-2}v|u|^{\alpha}+g(x)\;\;\text{in}\;\mathbb{R}^{N},\\
		& u, \, v >0\,  \mbox{ in }\,\mathbb{R}^{N},
	\end{aligned}
	\right.
\end{equation}
where $N>2s$, $\alpha,\,\beta>1$, $\alpha+\beta=2N/(N-2s)$, and $f,\, g$ are nonnegative functionals in the dual space of $\dot{H}^s(\mathbb{R}^{N})$, i.e., $\prescript{}{(\dot{H}^{s})'}{\langle}f,u{\rangle}_{\dot{H}^s}\geq 0$, whenever $u$ is a nonnegative function  in $\dot{H}^s(\mathbb{R}^{N})$. When $f=0=g$, we show that the ground state solution of \eqref{MAT1} is {\it unique}. On the other hand, when $f$ and $g$ are nontrivial nonnegative functionals with ker$(f)$=ker$(g)$, then we establish the existence of at least two different positive solutions of \eqref{MAT1} provided that $\|f\|_{(\dot{H}^s)'}$ and $\|g\|_{(\dot{H}^s)'}$  are small enough. Moreover, we also provide a global compactness result, which gives a complete description of the Palais-Smale sequences of the above system.
 \medskip

\end{abstract}

\maketitle

\section{Introduction}
	In this article we study existence,  uniqueness/multiplicity of positive
solutions to the following fractional nonhomogeneous elliptic system in $\Rn$
\begin{equation}
	\tag{$\mathcal S$}\label{MAT1}
	\left\{\begin{aligned}
		&(-\Delta)^s u = \frac{\al}{2_s^*}|u|^{\al-2}u|v|^{\ba}+f(x)\;\;\text{in}\;\mathbb{R}^{N},\\
		&(-\Delta)^s v = \frac{\ba}{2_s^*}|v|^{\ba-2}v|u|^{\al}+g(x)\;\;\text{in}\;\mathbb{R}^{N},\\
		& u, \, v >0\,  \mbox{ in }\,\mathbb{R}^{N},
	\end{aligned}
	\right.
\end{equation}
where $N>2s$, $\al,\,\ba>1$, $\al+\ba = 2^*_s:=2N/(N-2s)$, and $f,\, g$ are  nonnegative functionals in the dual space of $\dot{H}^s(\Rn)$. Here $(-\De)^s$ denotes the  fractional Laplace operator which can be defined for the Schwartz class functions $\mathcal{S}(\Rn)$  as follows
\begin{equation} \label{De-u}
  \left(-\Delta\right)^su(x): = c_{N,s}
\, \text{P.V.} \int_{\Rn}\frac{u(x)-u(y)}{|x-y|^{N+2s}} \, {\rm d}y, \quad c_{N,s}= \frac{4^s\Ga(N/2+ s)}{\pi^{N/2}|\Ga(-s)|}.
\end{equation}
Let
$$\dot{H}^s(R^{N}): =\bigg\{u\in L^{2^*_s}(\R^N) \; : \; \iint_{\mathbb{R}^{2N}}\frac{|u(x)-u(y)|^2}{|x-y|^{N+2s}}\,{\rm d}x\,{\rm d}y<\infty\bigg\},$$
be the homogeneous fractional Sobolev space, endowed 	with the inner product
$\langle\cdot,\cdot\rangle_{\dot{H}^s}$ and corresponding Gagliardo norm
$$\|u\|_{\dot{H}^{s}}:=\left( \iint_{\mathbb{R}^{2N}} \frac{|u(x)-u(y)|^2}{|x-y|^{N+2s}}\,{\rm d}x\,{\rm d}y\right)^{1/2}.$$
It is well-known that $u\in\dot{H}^s(\Rn)$ implies $u\in L^p_{\text{\scriptsize{\rm loc}}}(\Rn)$ for any $p\in[2,2^*_s]$.

In the vectorial case, as described in \cite{BCP}, the natural solution space for~\eqref{MAT1} is  the Hilbert space $\Hs\times\Hs$, equipped with the inner product
$$\big\langle (u,v), (\phi,\psi)\big\rangle_{\dot{H}^s\times\dot{H}^s}:=\langle u,\phi\rangle_{\dot{H}^s}+\langle v,\psi\rangle_{\dot{H}^s},$$ and the norm
$$\|(u,v)\|_{\dot{H}^s\times\dot{H}^s}:=\big(\|u\|^2_{\dot{H}^s}+\|v\|^2_{\dot{H}^s}\big)^\frac{1}{2}.$$
In general, given any two Banach spaces $X$ and $Y$, the product space $X\times Y$ is endowed with the following product norm
$$\|(x,y)\|_{X\times Y}:=\big(\|x\|_X^2+\|y\|_Y^2\big)^\frac{1}{2}.$$
For instance, $L^p(\Rn)\times L^p(\Rn)$ ($p>1$) is equipped with the product norm
$$\|(u,v)\|_{L^{p}\times L^{p}}:=\big(\|u\|^2_{L^p}+\|v\|^2_{L^p}\big)^\frac{1}{2}.$$

	\begin{definition}
{\rm		A pair $(u,v) \in \dot{H}^s(\Rn)\times\Hs$ is said to be a {\it positive} (weak) {\it solution of} \eqref{MAT1} if $u>0$ and $v>0$ in $\Rn$ and for every $(\phi,\psi)\in\dot{H}^s(\Rn)\times \Hs$ it holds
		\begin{align*}
\big\langle (u,v), (\phi,\psi)\big\rangle_{\dot{H}^s\times\dot{H}^s}&= \frac{\al}{2^*_s}\int_{\R{^N}}|u|^{\al-2}u|v|^{\ba}\phi\,{\rm d}x+\frac{\ba}{2^*_s}\int_{\R{^N}}|v|^{\ba-2}v|u|^{\al}\psi\,{\rm d}x\\
&\qquad\qquad\qquad+ \prescript{}{(\dot{H}^s)'}{\langle}f,\phi{\rangle}_{\dot{H}^s} +\prescript{}{(\dot{H}^s)'}{\langle}g,\psi{\rangle}_{\dot{H}^s},
\end{align*}
		where $\prescript{}{(\dot{H}^s)'}{\langle}\cdot,\cdot{\rangle}_{\dot{H}^s}$ denotes the duality bracket between the dual space $\dot{H}^s(\Rn)'$
of $\dot{H}^s(\Rn)$ and $\dot{H}^s(\Rn)$
itself.}
	\end{definition}

Define
\be\lab{S}
S=S_{\al+\ba}:=\inf_{u\in\dot{H}^s(\Rn)\setminus\{0\}}\frac{\|u\|^2_{\dot{H}^s}}
{\bigg(\displaystyle\int_{\Rn} |u|^{2^*_s}{\rm d}x\bigg)^{2/2^*_s}}\no\ee
and \be\lab{17-7-5}S_{\al,\ba}:=\displaystyle{\inf_{(u,v)\in \dot{H}^s\times\dot{H}^s\setminus\{(0,0)\}} \frac{\|u\|^2_{\dot{H}^s}+ \|v\|^2_{\dot{H}^s}}{\bigg(\displaystyle\int_{\Rn}|u|^{\al}|v|^{\ba}{\rm d}x\bigg)^{2/2^*_s}}}.\ee
In the celebrated paper \cite{CLO},  Chen, Li and Ou   prove that
the best Sobolev constant $S_{\al+\ba}=S$ is achieved by $w$, where $w$ is the unique positive solution (up to  translations and dilations) of
\be\lab{24-9-2}(-\Delta)^sw = w^{2^*_s-1}\;\;\text{in}\;\,\mathbb{R}^{N},\quad
w \in \dot{H}^{s}{(\mathbb{R}^{N})}.\ee

Next, we recall a result from \cite{FMPSZ} (\cite{AMS} in the local case) which states the relation between $S_{\al,\ba}$ and $S_{\al+\ba}$.

\begin{lemma}\lab{l:S}\cite[Lemma 5.1]{FMPSZ} In all cases $\al>1$, $\ba>1$, with $\al+\ba\le2^*_s$, it results
$$S_{\al,\ba}=\bigg[\left(\frac{\al}{\ba}\right)^\frac{\ba}{\al+\ba}
+\left(\frac{\al}{\ba}\right)^\frac{-\al}{\al+\ba} \bigg]S_{\al+\ba}.$$
Moreover, if $w$ achieves $S_{\al+\ba}$ then $(Bw, Cw)$ achieves $S_{\al,\ba}$ for all positive constants $B$ and $C$ such that $B/C=\sqrt{\al/\ba}$.
\end{lemma}

The scalar version of \eqref{MAT1} has been considered by Bhakta and Pucci  in \cite{BP}, where they prove  existence of at least two positive solutions.
This class of problems in the scalar and local cases, involving
Sobolev critical exponents was treated in the pioneering paper \cite{BNh}. Then existence was extended in \cite{Tarantello} to multiplicity results. These kind of problems were  studied in several directions. Let us mention  \cite{CaoZou, Castro, doO, Squassina, Wan} for more general perturbations and \cite{Clapp} for existence of sign changing solutions. Versions for systems were extended, for instance,  in \cite{Mohamed, Han, HSZ, Zhang} and in the references therein.

Elliptic systems arise in biological applications (e.g. population dynamics) or physical applications (e.g. models of a nuclear reactor) and have been drawn a lot of attention (see \cite{AMS, CFMT, M, RZ} and references therein). For systems in bounded domains with nonhomogeneous terms we refer to~\cite{BR}. Problems involving the fractional Laplace operator
appear in several areas such as phase transitions, flames propagation, chemical reaction in liquids, population dynamics, finance, etc.,  see  for e.g. \cite{Caffarelli,Nezza}.

In the nonlocal case, there are not so many papers, in which weakly coupled systems of equations have been studied. We refer to~\cite{BN, CS, CMSY, FMPSZ, GMS, HSZ}, where Dirichlet systems of equations in bounded domains have been treated. For the nonlocal systems of equations in the entire space $\Rn$, we cite \cite{BCP, FPS, FPZ}
and the references therein. In the very recent work \cite{BCP}, the first, second and fourth authors of this current paper have proved existence of one solution to \eqref{MAT1} when $f$ and $g$ are nontrivial but $\|f\|_{(\dot{H}^s)'}$ and $\|g\|_{(\dot{H}^s)'}$ are small enough. To the best of our knowledge, so far there have been no papers in the literature, where uniqueness/multiplicity of positive solutions  have  been established for \eqref{MAT1}, with the fractional Laplacian and the critical exponents in $\Rn$. The main results in the paper are new even in the local case $s=1$.

\medskip

 First of all, we say that  a pair $(u,v)\in \dot{H}^s(\Rn)\times\Hs$ is a {\it ground state solution} or {\it least energy solution} for \eqref{MAT1}, with $f=0=g$, if $(u,v)$ is a minimizer of $S_{\alpha,\beta}.$

 \medskip

 Lemma~\ref{l:S} poses a natural question: are all the ground state solutions of \eqref{MAT1}, with $f=0=g$,  of the form $(Bw, Cw)$, where $w$ is the unique positive solution of \eqref{24-9-2}?

We answer this question affirmatively in our first main theorem which is stated as below.

 \medskip

\begin{theorem}[Uniqueness of ground state for homogeneous system]\lab{th:uni}
 Let $(u_0,v_0)$ be a  minimizer of  $S_{\alpha,\beta}.$  Then  there exist $\tau, B>0$  such that
	$$(u_0,v_0)=(Bw,Cw),\ \mbox{with }\  C=B\tau,  \quad \tau=\sqrt\frac{\ba}{\al},$$
	where $w$ is the unique positive solution of \eqref{24-9-2}.
	\end{theorem}
	
The above result partially extends the uniqueness theorem due to  Chen, Li and Ou \cite{CLO} from the scalar case \eqref{24-9-2} to the system  \eqref{MAT1} with $f=0=g$.  Theorem~\ref{th:uni} proves the uniqueness of ground state solution of the system~\eqref{MAT1} when $f=0=g$ and also generalizes \cite[Lemma 5.1]{FMPSZ}, where as in \cite{CLO} uniqueness has been established among all positive solutions of \eqref{24-9-2}.	

Our next main result is the multiplicity of solutions for the nonhomogeneous system \eqref{MAT1}.
	
	\begin{theorem}[Multiplicity for nonhomogeneous system]\lab{th:ex-f}
Assume that  $f, g$ are  nontrivial nonneagtive functionals in the dual space of $\Hs$ with $ker(f)=ker(g)$ and
$$ \max\{\|f\|_{(\dot{H}^s)'}, \|g\|_{(\dot{H}^s)'}\} <C_0S_{\al,\ba}^{\frac{N}{4s}}, \quad\text{where}\quad C_0:=\bigg(\frac{4s}{N+2s}\bigg)(2^*_s-1)^{-\frac{N-2s}{4s}},$$
then \eqref{MAT1} admits at least two positive solutions.

Furthermore, if $f\equiv g$, then the solution $( u, v)$ of \eqref{MAT1} has the property that
$ u\not\equiv v$, whenever $\al\neq\ba$. Finally, if $\al=\ba$ but $f\not\equiv g$, then $u\not\equiv  v$.
\end{theorem}
Theorem~\ref{th:ex-f}  complements
the mentioned work \cite{BCP} on \eqref{MAT1}.


\medskip

The proof of the uniqueness Theorem~\ref{th:uni} is inspired by some arguments made in \cite{CZ-1} and \cite{PPW} (also see \cite{CZ}). The main difference is that in our case the nontrivial solution $(u,v)$ has both components nontrivial, that is $u\neq 0$ and $v \neq 0$, and in the proof it was necessary to deal with a non symmetric system. 

To prove the multiplicity Theorem~\ref{th:ex-f}, the main difficulty is  the lack of compactness of the Sobolev space $\dot{H}^s(\Rn)$  into the Lebesgue space $L^{2^*_s}(\Rn)$.  For this reason   the functional associated to  system \eqref{MAT1}  may fail to satisfy the Palais-Smale condition at some critical levels. To overcome this, it is necessary to look for a nice energy range where the $(PS)$ condition  holds in order to use variational arguments. Classification of $(PS)$ sequences associated with a scalar equation (local/nonlocal) has been done  in many papers, to quote a few, we cite \cite{BP, CF, Lions, PS, PS-2, S}. To the best of our knowledge, the $(PS)$ decomposition associated to systems of equations has not been studied much. We quote the recent work \cite{PPW}, where in the local case the $(PS)$ decomposition was done for systems of equations in bounded domains.

Again to the best of our knowledge, in both the local and nonlocal cases, Proposition~\ref{PSP} (see, Section~3) is the first result where the $(PS)$  decomposition has been established for system of equations in the whole space $\Rn$.
Next, to prove multiplicity of solutions, we decompose the space $\dot{H}^s(\Rn)$ into three disjoint components.  The first solution is constructed using a minimization argument in one of the components. Another solution is obtained by combining the Ekeland variational  principle with a careful analysis of the critical levels by using the homogeneous unique solution with some estimates in a slightly larger Morrey space.

The paper has been organized as follows. In Section 2, we prove the uniqueness for the ground state solution of the homogeneous system, namely Theorem~\ref{th:uni}. Section 3 deals with the Palais-Smale decomposition associated with the functional of \eqref{MAT1}.  In Section 4, we prove Theorem~\ref{th:ex-f}. In the Appendix we discuss a few elementary embeddings of the product of Morrey spaces.

\bigskip

{\bf Notation:} $u_+:=\max\{u,0\},\;u_-:=-\min\{u,0\}$. By ker$(f)$ we denote the kernel of $f$.
\bigskip

\begin{remark} {\rm
Adapting the arguments in the proof of Theorem~\ref{th:uni} and Theorem~\ref{th:ex-f}, the results of uniqueness and multiplicity can be obtained for the following system of equations:
\begin{itemize}
\item [a)] \begin{equation}
\label{sys-Q'}
\left\{\begin{aligned}
		&(-\Delta)^s u + u= \frac{\al}{\al+\ba}|u|^{\al-2}u|v|^{\ba}+f(x)\;\;\text{in}\,\, \mathbb{R}^{N},\\
		&(-\Delta)^s v + v= \frac{\ba}{\al+\ba}|v|^{\ba-2}v|u|^{\al}+g(x)\;\;\text{in}\,\,  \mathbb{R}^{N},\\
              & u, \, v >0\,  \mbox{ in }\,\,\mathbb{R}^{N},
		 \end{aligned}
  \right.
\end{equation}
where $N>2s$, $\al,\,\ba>1$ and $\al+\ba<2^*_s$, and $f,\, g$ are nonnegative functionals in the dual space of ${H}^s(\Rn)$ (see \cite{BCP} for existence of solutions). It is known that the scalar equation
\be\lab{29-9-1}
(-\Delta)^s u + u=|u|^{\al+\ba-2}u \quad \mbox{ in }\,\mathbb{R}^{N}
\ee
has a unique ground state solution (see \cite{FLS}). If $\om$ denotes the unique ground state solution of \eqref{29-9-1}, then it can be shown that $(r\om, t \om)$ is a ground state solution of \eqref{sys-Q'} when $f=0=g$ and $r/t=\sqrt{\al/\ba}$. Next, following an argument similar to Theorem~\ref{th:uni}, with obvious modifications, it can be shown that any ground state solution of \eqref{sys-Q'} with $f=0=g$ is of the form $(r\om, t \om)$ where $r/t=\sqrt{\al/\ba}$.
\item [b)]  \begin{equation}
	\left\{\begin{aligned}
		&(-\Delta)^s u = \frac{\al}{2_s^*}a(x)|u|^{\al-2}u|v|^{\ba}+f(x)\;\;\text{in}\,\, \mathbb{R}^{N},\\
		&(-\Delta)^s v = \frac{\ba}{2_s^*}b(x)|v|^{\ba-2}v|u|^{\al}+g(x)\;\;\text{in}\,\, \mathbb{R}^{N},\\
		& u, \, v >0\,  \mbox{ in }\,\mathbb{R}^{N},
	\end{aligned}
	\right.
\end{equation}
where $\al,\,\ba$, $f,\, g$ are as in \eqref{MAT1} and the potentials $a,\, b$ are continuous functions in $\Rn$ with $a, b\geq 1$ and $a(x), \, b(x)\to 1$ as $|x|\to \infty$. See for instance \cite{BP} in the scalar case.

\item [c)]  One can also try to adopt the methodology of this paper in order to study the system of equations involving  the  Hardy operator  i.e., if  $(-\De)^s$ is replaced by the Hardy operator $(-\De)^s-\frac{\ga}{|x|^{2s}}$, where $\ga\in(0, \ga_{N,s})$ and $\ga_{N,s}$ is the best Hardy constant in the fractional Hardy inequality. The multiplicity question in the scalar case  was
    already solved for this problem in the recent paper~\cite{BPh}.
\end{itemize}
}
\end{remark}
\begin{remark} {\rm
Theorem~\ref{th:uni} proves uniqueness of ground state solutions of \eqref{MAT1} with $f=0=g$. Therefore, it is interesting to ask if any positive solution of  \eqref{MAT1} with $f=0=g$ is of the form $(rw, tw)$, where $r/t=\sqrt{\al/\ba}$ and $w$ is the {\it unique} positive solution of \eqref{24-9-2}.}
\end{remark}

\section{Uniqueness for the homogeneous system}	
	
 First we need an auxiliary lemma which will be used to prove Theorem~\ref{th:uni}. Consider the following system with a parameter $\mu>0$

  \begin{equation}
	\label{MAT2}
	\left\{\begin{aligned}
		&(-\Delta)^s u = \mu  |u|^{2^{*}_s -2} u + \frac{ \al}{2_s^*}|u|^{\al-2}u|v|^{\ba}\;\;\text{in}\;\mathbb{R}^{N},\\
		&(-\Delta)^s v = \frac{\ba}{2_s^*}|v|^{\ba-2}v|u|^{\al}\;\;\text{in}\;\mathbb{R}^{N},\\
		& u, \, v >0\,  \mbox{ in }\,\mathbb{R}^{N} .
	\end{aligned}
	\right.
\end{equation}

\noindent Associated to \eqref{MAT2}, we define
\be\lab{24-9-1}
S_{\mu,\al,\ba}:=\displaystyle{\inf_{(u,v)\in \dot{H}^s\times\dot{H}^s\setminus\{(0,0)\}} \frac{\|(u,v)\|^2_{\phs}}{\bigg(\displaystyle\mu\int_{\Rn}|u|^{2^*_s}{\rm d}x+\int_{\Rn}|u|^{\al}|v|^{\ba}{\rm d}x\bigg)^{2/2^*_s}}}.
\ee

  \begin{lemma}\lab{l:gss}
  (i) Let $h(\tau):=\displaystyle\frac{1+\tau^2}{(\mu +\tau^\beta)^{2/2^{*}_s}}, \,\, \tau>0$. Then there exists $\mu_0>0$ such that for $\mu\in(0,\mu_0)$,
 $$S_{\mu,\al,\ba}=h(\tau_0)S, \quad\text{where}\quad h(\tau_0)=\min_{\tau>0}h(\tau).$$
Furthermore, $\tau_0=\tau_0(\mu, \al,\ba, N,s)>0$.

(ii) For any $r>0$, $(rw, r\tau_0 w)$ achieves $S_{\mu,\al,\ba}$, where $w$ is the unique positive solution of \eqref{24-9-2}.
  \end{lemma}

 \begin{proof}
 Let $\{(u_n, v_n)\}$ be a minimizing sequence for $S_{\mu,\al,\ba}$. Choose $\tau_n>0$ such that $\|v_n\|_{L^{2^*_s}(\Rn)}=\tau_n\|u_n\|_{L^{2^*_s}(\Rn)}$. Now set, $\displaystyle z_n=\frac{v_n}{\tau_n}$. Therefore, $\|u_n\|_{L^{2^*_s}(\Rn)}=\|z_n\|_{L^{2^*_s}(\Rn)}$ and applying Young's inequality,
 $$\int_{\Rn}|u_n|^\al|z_n|^{\ba}{\rm d}x\leq \frac{\al}{2^*_s}\int_{\Rn}|u_n|^{2^*_s}{\rm d}x+\frac{\ba}{2^*_s}\int_{\Rn}|z_n|^{2^*_s}{\rm d}x=\int_{\Rn}|u_n|^{2^*_s}{\rm d}x=\int_{\Rn}|z_n|^{2^*_s}{\rm d}x.$$
 Hence,
 \Bea
 S_{\mu,\al,\ba}+o(1)&=&\frac{\|u_n\|^2_{\dot{H}^s}+\|v_n\|^2_{\dot{H}^s}}{\bigg(\displaystyle\mu\int_{\Rn}|u_n|^{2^*_s}{\rm d}x+\int_{\Rn}|u_n|^{\al}|v_n|^{\ba}{\rm d}x\bigg)^{2/2^*_s}}\\
 &=&\frac{\|u_n\|^2_{\dot{H}^s}}{\bigg(\displaystyle\mu\int_{\Rn}|u_n|^{2^*_s}{\rm d}x+\tau_n^\ba\int_{\Rn}|u_n|^{\al}|z_n|^{\ba}{\rm d}x\bigg)^{2/2^*_s}}\\
&&\quad +\frac{\tau_n^2\|z_n\|^2_{\dot{H}^s}}{\bigg(\displaystyle\mu\int_{\Rn}|z_n|^{2^*_s}{\rm d}x+\tau_n^\ba\int_{\Rn}|u_n|^{\al}|z_n|^{\ba}{\rm d}x\bigg)^{2/2^*_s}}\\
&\geq&\frac{1}{(\mu+\tau_n^\ba)^{2/2^*_s}}\frac{\|u_n\|^2_{\dot{H}^s}}
{\big(\int_{\Rn} |u_n|^{2^*_s}{\rm d}x\big)^{2/2^*_s}}+\frac{\tau_n^2}{(\mu+\tau_n^\ba)^{2/2^*_s}}\frac{\|z_n\|^2_{\dot{H}^s}}
{\big(\int_{\Rn} |z_n|^{2^*_s}{\rm d}x\big)^{2/2^*_s}}\\
&\geq&\frac{1+\tau_n^2}{(\mu +\tau_n^\beta)^{2/2^{*}_s}} S\geq \min_{\tau>0}h(\tau)S.
 \Eea
 Note that $h$ is a $C^1$ function with  $h(\tau)>0$ for all $\tau\geq 0$, $h(\tau)\to \infty$ as $\tau\to\infty$ and $h(\tau)\to\mu^{-\frac{2}{2^*_s}}$ as $\tau\to 0$. Therefore, there exists $\tau_0\geq 0$ such that $\min_{\tau>0}h(\tau)=h(\tau_0)$. Next, we claim that $\tau_0>0$, if we choose $\mu>0$ small enough. To prove the claim, first we note that $h(0)=\mu^{-\frac{2}{2^*_s}}$ and
 $h(1)=2(1+\mu)^{-{2/2^*_s}}$. Therefore, we can choose $\mu_0>0$ small enough such that for $\mu\in(0,\mu_0)$,  $h(0)>h(1)$. Thus, $h$ can not attain global minimum at $0$, if $\mu\in(0,\mu_0)$. Hence $\tau_0>0$.

Consequently, $S_{\mu,\al,\ba}+o(1)\geq h(\tau_0)S,$ and as $o(1)\to 0$ as $n\to \infty$, we get $S_{\mu,\al,\ba}\geq h(\tau_0)S$. On the other hand, choosing $(u,v)=(w, \tau_0 w)$, we easily see that
 $S_{\mu,\al,\ba}\leq h(\tau_0)S$. Hence $S_{\mu,\al,\ba}=h(\tau_0)S$.

Since $\tau_0$ is the minimum point for $h$, clearly $h'(\tau_0)=0$. Thus $\tau_0$ satisfies $$\tau\big(\mu  2^{*}_s + \alpha \tau^\beta - \beta \tau^{\beta-2}\big)=0.$$
But $\tau_0>0$, and so $\tau_0$ satisfies $\mu  2^{*}_s + \alpha \tau^\beta - \beta \tau^{\beta-2}=0$. This proves (i).

(ii) Note that for $(u,v)=(rw, r\tau_0 w)$, an easy computation yields
 $$ \frac{\|(u,v)\|^2_{\phs}}{\bigg(\displaystyle\mu\int_{\Rn}|u|^{2^*_s}{\rm d}x+\int_{\Rn}|u|^{\al}|v|^{\ba}{\rm d}x\bigg)^{2/2^*_s}}=h(\tau_0)S.$$
 Hence using (i), we conclude that $S_{\mu,\al,\ba}$ is achieved by $(rw, r\tau_0 w)$.
  \end{proof}

 \medskip

{\bf Proof of Theorem~\ref{th:uni}}\begin{proof}
Suppose that $(u_0,v_0)$  and $w$ achieves $S_{\alpha,\beta}$  and $S$ respectively. We are going to prove that there are $r, t>0$ such that
$$(u_0,v_0)=(r w, t w).$$

	
	\noindent
	
	 {\bf Claim. }  \begin{description}\item[a)] $$\int_{\R^N}  |u_0|^{\alpha}|v_0|^{\beta} \,{\rm d}x = r^{\alpha} t^{\beta} \int_{\R^N}  w^{ 2^{*}_s}\,{\rm d}x, \quad\text{ whenever}\quad \frac{r}{t}=\sqrt{\frac{\al}{\ba}}.$$
	\item[b)] There exists $r>0$ such that $$\int_{\R^N}  |u_0|^{2^{*}_s}\,{\rm d}x = r^{2^{*}_s} \int_{\R^N}  w^{ 2^{*}_s}\,{\rm d}x. $$
	\end{description}
	Assuming the Claim for a while,  first we complete the proof.
	
	Indeed, fix $r$ as found in claim b) and set $t=r\sqrt{\ba/\al}$. Therefore, by Lemma~\ref{l:S}, $(rw, tw)$ achieves $S_{\al,\ba}$. Consequently, $(rw, tw)$ solves \eqref{MAT1} with $f=0=g$ and
	\begin{equation}\label{igual} \frac{\al}{2^{*}_s } r^{\alpha-2}t^{\beta}= 1= \frac{\ba}{2^{*}_s}  r^{\alpha}t^{\beta-2}.\end{equation}
Now define $(u_1,v_1)=(\frac{ u_0}{r},\frac{v_0}{t}) .$
Then, by Claim a)  we have
\begin{equation*}
\|u_1\|_{\dot{H}^{s}}^2 = \frac{1}{r^2}\|u_0\|_{\dot{H}^{s}}^2=\frac{\alpha}{ 2^{*}_s  r^2} \int_{\R^N}  |u_0|^{\alpha}|v_0|^{\beta} \,{\rm d}x =\frac{\alpha r^{\alpha} t^{ \beta}}{ 2^{*}_s  r^2} \int_{\R^N}  |w|^{2^{*}_s } \,{\rm d}x = \|w\|_{\dot{H}^{s}}^2,  \end{equation*}
where for the last equality we have used (\ref{igual}). Similarly, it follows that
$$ \|v_1\|_{\dot{H}^{s}}^2= \|w\|_{\dot{H}^{s}}^2. $$
Therefore
\begin{equation}\label{(*)}
 \|u_1\|_{\dot{H}^{s}}^2= \|w\|_{\dot{H}^{s}}^2= \|v_1\|_{\dot{H}^{s}}^2.
\end{equation}
\noindent
Further, using Claim b) in the definition of $u_1$ yields
\begin{equation}\label{(**)}
\int_{\R^N}  |u_1|^{2^{*}_s} \,{\rm d}x =\int_{\R^N}  | w|^{2^{*}_s } \,{\rm d}x.
\end{equation}
Combining (\ref{(*)}) and (\ref{(**)}), by the uniqueness result in the scalar case, see  \cite{CZ},  we conclude that
$$u_1=w, \ \mbox{that is}\  u_0=r w.$$
Now we prove that $v_1=w$.  Indeed, by Claim a)
\Bea\int_{\Rn}| w|^{2^{*}_s } \,{\rm d}x=\int_{\Rn}|u_1|^\al|v_1|^\ba{\rm d}x&\leq&\bigg(\int_{\Rn}|u_1|^{2^*_s}{\rm d}x\bigg)^{\al/2^*_s}\bigg(\int_{\Rn}|v_1|^{2^*_s}{\rm d}x\bigg)^{\ba/2^*_s}\\
&=&\bigg(\int_{\Rn}|w|^{2^*_s}{\rm d}x\bigg)^{\al/2^*_s}\bigg(\int_{\Rn}|v_1|^{2^*_s}{\rm d}x\bigg)^{\ba/2^*_s}. \Eea
Consequently, $\|w\|_{L^{2^*_s}}\leq \|v_1\|_{L^{2^*_s}}$. Combining this with \eqref{(*)} and the fact that $w$ achieves $S$, we obtain
$$S^{-{1/2}}\|v_1\|_{\dot{H}^s}=S^{-{1/2}}\|w\|_{\dot{H}^s}=\|w\|_{L^{2^*_s}}\leq \|v_1\|_{L^{2^*_s}}\leq S^{-{1/2}}\|v_1\|_{\dot{H}^s}.$$
Hence the inequality becomes equality in the above expression, i.e., $v_1$ achieves $S$. Again by the uniqueness result in the scalar case, we conclude that
$$v_1=w, \ \mbox{that is}\  v_0=t w.$$

\noindent
This proves Theorem \ref{th:uni}. Now we are going to prove the Claim. First, we prove {\bf Claim a)}.

\medskip

Consider the following problem with a parameter $\mu>0$
 \begin{equation}
	\tag{$\mathcal S_\mu$}\label{MAT1mu}
	\left\{\begin{aligned}
		&(-\Delta)^s u = \frac{\mu \al}{2_s^*}|u|^{\al-2}u|v|^{\ba}\;\;\text{in}\;\mathbb{R}^{N},\\
		&(-\Delta)^s v = \frac{\mu\ba}{2_s^*}|v|^{\ba-2}v|u|^{\al}\;\;\text{in}\;\mathbb{R}^{N},\\
		& u, \, v >0\,  \mbox{ in }\,\mathbb{R}^{N}.
	\end{aligned}
	\right.
\end{equation}
Associated to \eqref{MAT1mu}, define the following min-max problem
$$B(\mu):=\inf_{(u,v)\in \dot{H}^s\times \dot{H}^s\setminus\{(0,0)\}} \max_{t > 0} E_\mu( t u,t v),$$
where
$$E_\mu(u,v):=\frac{1}{2}\|(u,v)\|^2_{ \dot{H}^s\times \dot{H}^s}-\frac{\mu}{2^{*}_s}\int_{\R^N}|u|^{\al}|v|^{\ba}\,{\rm d}x.$$
\noindent
Note that there exists $t_\mu> 0$ such that
$$\max_{t >0} E_\mu (t u_0, t v_0)=E_\mu (t_\mu u_0, t_\mu v_0),$$
where $t_\mu$ satisfies
$$H(\mu,t_\mu ) =0 \quad  \mbox{and}\quad  H(\mu,t):= C- \mu D t^{ 2^{*}_ s -2},$$
with
$$ C= \|(u_0,v_0)\|^{2}_{\dot{H}^s\times\dot{H}^s} \ \mbox{and}\ D=\int_{\R^N}|u_0|^{\al}|v_0|^{\ba}\,{\rm d}x.$$
Since $(u_0,v_0)$ is a least energy solution of \eqref{MAT1} with $f=0=g$,
$$H(1,1)=0,\ \frac{\partial}{\partial t} H(1,1)< 0 \ \mbox{and}\  H(\mu,t_u )=0.$$
By the implicit function theorem $t_\mu$ is a $C^1$ function near of $\mu=1,$ and
$$t^{\prime}_\mu \mid_{\mu=1}= -\frac{ \frac{\partial}{\partial \mu} H}{\frac{\partial}{\partial t} H}\bigg|_{\mu=1=t} =-\frac{1}{(2^{*}_s -2)} .$$
By the Taylor formula around $\mu=1$,  we have
 $$t_\mu(\mu)= 1 + t_\mu^{\prime}(1)(\mu-1) + O(|\mu-1|^2)$$
Consequently, $$t_\mu^2(\mu)= 1 + 2t_\mu^{\prime}(1)(\mu-1) + O(|\mu-1|^2).$$
Further, as $H(\mu,t_ \mu)= C- \mu D t^{ 2^{*}_ s -2} _{\mu}=0$ and $C=D$, we have $t^{ 2^{*}_ s -2} _{\mu}=\mu^{-1}$. Therefore,
 \begin{eqnarray}\lab{14-9-2}
 B(\mu) \leq E_\mu(t_\mu u_0, t_\mu v_0)
 &=&t^{2}_{\mu}\bigg(\frac{1}{2}-\frac{\mu t_\mu^{2^*_s-2}}{2^*_s}\bigg)\|(u_0,v_0)\|^2_{\phs}\no\\
 &=& t^{2}_{\mu}\frac{s}{N}\|(u_0,v_0)\|^2_{\phs}= t^{2}_{\mu} B(1)\no\\
 &=& B(1) - \frac{2}{(2^*_s -2)} B(1) (\mu -1) + O(|\mu -1|^2).
 \end{eqnarray}
From the definition of $B(1)$,  a direct computation yields
\begin{align}\lab{14-9-3}
B(1)&=\inf_{(u,v)\in \dot{H}^s\times \dot{H}^s\setminus\{(0,0)\}}E_1(\tilde t u, \tilde tv), \quad\text{where}\quad \tilde t=\bigg(\frac{\|(u,v)\|_{\phs}^2}{\int_{\Rn}|u|^\al|v|^\ba{\rm d}x}\bigg)^{1/(2^*_s-2)}\no\\
&=\inf_{(u,v)\in \dot{H}^s\times \dot{H}^s\setminus\{(0,0)\}}\frac{s}{N}\bigg(\frac{\|(u,v)\|_{\phs}^2}{\big(\int_{\Rn}|u|^\al|v|^\ba{\rm d}x\big)^{2/2^*_s}}\bigg)^{2^*_s/(2^*_s-2)}\no\\
&=\frac{s}{N}S_{\al,\ba}^\frac{2^*_s}{2^*_s-2}=\frac{s}{N}\bigg(\frac{\|(u_0,v_0)\|_{\phs}^2}{\big(\int_{\Rn}|u_0|^\al|v_0|^\ba{\rm d}x\big)^{2/2^*_s}}\bigg)^{2^*_s/(2^*_s-2)}=\frac{sD}{N}.
\end{align}
Substituting the above value of $B(1)$ in \eqref{14-9-2} yields
$$B(\mu)\leq B(1) - \frac{D}{2^{*}_s  } (\mu -1) + O(|\mu -1|^2).$$
Thus
\begin{equation}\lab{6-10-1}
 \frac{B(\mu) -B(1)}{\mu -1}\begin{cases}
\leq - \frac{D}{2^{*}_s  }+ O(|\mu -1|) \quad\text{if}\quad \mu>1\\
\geq - \frac{D}{2^{*}_s  }+ O(|\mu -1|)\quad\text{if}\quad \mu<1.
\end{cases}
\end{equation}
The first inequality in \eqref{6-10-1} implies $\displaystyle B'(1)\leq - \frac{D}{2^{*}_s  }$ and the second inequality in \eqref{6-10-1} implies $\displaystyle B'(1)\geq - \frac{D}{2^{*}_s  }$.
 Hence,
\be\lab{14-9-4}B^{\prime}(1)=  - \frac{D}{2^{*}_s  }= - \frac{1}{2^{*}_s  }\int_{\R^N}|u_0|^{\al}|v_0|^{\ba}\,{\rm d}x.\ee
On the other hand, proceeding as in \eqref{14-9-3}, we derive that
$$B(\mu)=\frac{s}{N}\frac{1}{\mu^\frac{2}{2^*_s-2}}\inf_{(u,v)\in \dot{H}^s\times \dot{H}^s\setminus\{(0,0)\}}\bigg(\frac{\|(u,v)\|_{\phs}^2}{\big(\int_{\Rn}|u|^\al|v|^\ba{\rm d}x\big)^\frac{2}{2^*_s}}\bigg)^{2^*_s/(2^*_s-2)}=\frac{s}{N}\frac{1}{\mu^\frac{2}{2^*_s-2}}S_{\al,\ba}^\frac{2^*_s}{2^*_s-2}.$$
Since, $(r w,t w)$ (for any $r, t>0$ with $r/t=\sqrt{\al/\ba}$) is also a  ground state solution of (\ref{MAT1}) with $f=0=g$, from the above expression of $B(\mu)$, we obtain
  $$ B(\mu)=\frac{s}{N}\frac{1}{\mu^\frac{2}{2^*_s-2}}r^{\alpha}t^{\beta} \int_{\Rn}|w|^{2^*_s}{\rm d}x \quad\implies\quad
B^{\prime}(1)=- \frac{r^{\alpha}t^{\beta} }{2^{*}_s  }\int_{\R^N}|w|^{2^{*}_s}\,{\rm d}x.$$
Comparing this with \eqref{14-9-4}, we conclude that
  $$\int_{\R^N}|u_0|^{\al}|v_0|^{\ba}\,{\rm d}x= r^{\alpha}t^{\beta} \int_{\R^N}|w|^{2^{*}_s}\,{\rm d}x,$$
  where $r, t>0$ are arbitrary with $r/t=\sqrt{\al/\ba}$. This proves Claim a).

  \medskip

  Let us turn to the proof of {\bf Claim b)}. Let $\mu_0$ be as in Lemma~\ref{l:gss}. Consider the system \eqref{MAT2} with $\mu\in (0,\mu_0)$ and define the following min-max problem
$$\tilde B(\mu):=\inf_{(u,v)\in \dot{H}^s\times\dot{H}^s\setminus\{(0,0)\}} \max_{t > 0} \tilde E_\mu( t u,t v),$$
where
$$\tilde E_\mu(u,v):=\frac{1}{2}\|(u,v)\|^2_{\Hs\times\Hs}-\frac{\mu}{2^{*}_s}\int_{\R^N} |u|^{ 2^{*}_s}\,{\rm d}x - \frac{1}{2^{*}_s}\int_{\R^N}|u|^{\al}|v|^{\ba}\,{\rm d}x.$$
\noindent
Note that there exists $t_\mu> 0$ such that
$$\max_{t >0} \tilde E_\mu (t u_0, t v_0)=\tilde E_\mu (t_\mu u_0, t_\mu v_0),$$
where $t_\mu$ satisfies
$$\tilde H(\mu,t_\mu ) =0 \  \mbox{and}\  \tilde H(\mu,t)= C-( \mu G+ D) t^{ 2^{*}_ s -2}$$
with
$$  C= \|(u_0,v_0)\|^{2}_{\dot{H}^s\times\dot{H}^s}\quad  G=\int_{\R^N}|u_0|^{2^{*}_s}\,{\rm d}x \quad \mbox{and}\quad  D=\int_{\R^N}|u_0|^{\al}|v_0|^{\ba}\,{\rm d}x.$$
Since $(u_0,v_0)$ is a ground state solution of \eqref{MAT1} with $f=0=g$,
$$\tilde H(0,1)=C-D= 0,\quad \frac{\partial}{\partial t} \tilde H(0,1) =-(2^{*}_s -2)D\quad \mbox{and}\quad  \frac{\partial}{\partial \mu}\tilde H(0,1)=-G, $$
evaluated at $t=1$ and $\mu=0.$

By the implicit function theorem $t_\mu$ is a $C^1$ function near of $\mu=0,$ and
$$t^{\prime}_\mu \mid_{\mu=0}= -\frac{ \frac{\partial}{\partial \mu} \tilde H}{\frac{\partial}{\partial t}\tilde H}\mid_{\mu=0,t=1} =-\frac{G}{(2^{*}_s -2)D} .$$
The Taylor formula around $\mu=0$ and $t_\mu=1$  yields
 $$t_\mu(\mu)= 1 +\mu t_\mu^{\prime}(0) + O(|\mu|^2),$$
consequently, $$t_\mu^2(\mu)= 1 + 2\mu t_\mu^{\prime}(0) + O(|\mu|^2).$$
Now  $\tilde B(0)=B(1)$, where $B(.)$ is as defined in the proof of Claim a). Therefore,  $\tilde B(0)=\displaystyle\frac{sD}{N}$.

Since $\tilde H(\mu,t_ \mu)= C-( \mu G+D) t^{ 2^{*}_ s -2} =0$, and $C=D$ using an argument as before, it follows that

 \begin{eqnarray*}
 \tilde B(\mu) \leq \tilde E_\mu(t_\mu u_0, t_\mu v_0)
 =\frac{t^{2}_{\mu}}{2}C-\frac{t^{2^*_s}_{\mu}}{2^*_s}(\mu G+D)
 &=& t^{2}_{\mu}\frac{sD}{N}=t^{2}_{\mu}\tilde B(0)\\
 &=& \tilde B(0) - \frac{2G}{(2^{*}_s -2)D}  \mu \tilde B(0) + O(|\mu |^2)\\
 &=&\tilde B(0) - \frac{1}{2^{*}_s  } G\mu + O(|\mu|^2).
 \end{eqnarray*}
Then
\be\lab{16-9-3}\tilde B^{\prime}(0)= \lim_{\mu \rightarrow 0} \frac{\tilde B(\mu) -\tilde B(0)}{\mu } = - \frac{G}{2^{*}_s  }=  - \frac{1}{2^{*}_s  }\int_{\R^N}|u_0|^{2^{*}_s}\,{\rm d}x.\ee
On the other hand,  from the definition of $\tilde B(\mu)$, a straight forward computation yields

\begin{align}
 \tilde B(\mu)&=\inf_{(u,v)\in \phs\setminus\{(0,0)\}} \tilde E_\mu (\tilde t u,\tilde t v) \quad\text{where}\quad \tilde t= \bigg(\frac{\|(u,v)\|_{\phs}^2}{\mu\int_{\Rn}|u|^{2^*_s}{\rm d}x+\int_{\Rn}|u|^\al|u|^\ba{\rm d}x}\bigg)^{1/(2^*_s-2)}\no\\
 &=\frac{s}{N}\inf_{(u,v)\in \phs\setminus\{(0,0)\}}\bigg[\frac{\|(u,v)\|^2}{\big(\mu\int_{\Rn}|u|^{2^*}{\rm d}x+\int_{\Rn}|u|^\al|v|^\ba{\rm d}x\big)^{2/2^*_s}}\bigg]^{2^*_s/(2^*_s-2)}\no .
\end{align}
Since by Lemma~\ref{l:gss}, $S_{\mu,\al,\ba}$ is achieved by $( r w, \tau_0 r w)$,  an easy computation yields
 $$ \tilde B(\mu)=\frac{s}{N}\bigg(\frac{1+\tau_0^2}{ (\mu +\tau_0^\beta)^{2/2^{*}_s} }\bigg)^{2^*_s/(2^*_s-2)}\int_{\Rn}|u|^{2^*}{\rm d}x.$$
As a consequence, $$ \tilde B'(0)=-\frac{1}{2^*_s}\bigg(\frac{1+\tau_0^2}{\tau_0^{\ba}}\bigg)^{2^*_s/(2^*_s-2)}\int_{\R^N}|w|^{2^{*}_s}\,{\rm d}x.$$
Now set  $$\tilde r=\bigg(\frac{1+\tau_0^2}{\tau_0^{\ba}}\bigg)^{1/(2^*_s-2)}.$$
Therefore, $\tilde B'(0)=-\displaystyle\frac{\tilde r^{2^*_s}}{2^*_s}\displaystyle\int_{\R^N}|w|^{2^{*}_s}\,{\rm d}x.$ Comparing this with \eqref{16-9-3} yields
$$\int_{\R^N}|u_0|^{2^{*}_s}\,{\rm d}x= \tilde r^{2^*_s}\int_{\R^N}|w|^{2^{*}_s}\,{\rm d}x.$$
This proves Claim b). Thus, we conclude the proof of Theorem~\ref{th:uni}.
 \end{proof}

\section{The Palais-Smale  decomposition}

In this section we study the Palais-Smale sequences (in short, $(PS)$ sequences) of the functional associated to \eqref{MAT1}, namely,
\be\label{EF}
			I_{f,g}(u,v):= \frac{1}{2}\|(u,v)\|^2_{\dot{H}^s\times\dot{H}^s}-\frac{1}{2^*_s}\int_{\R^N}|u|^{\al}|v|^{\ba}\,{\rm d}x -\prescript{}{(\dot{H}^s)'}{\langle}f,u{\rangle}_{\dot{H}^s}-\prescript{}{(\dot{H}^s)'}{\langle}g,v{\rangle}_{\dot{H}^s}.
		\ee
We say that the sequence $\{(u_n,v_n)\}\subset \dot{H}^s(\R^N)\times\Hs$ is a $(PS)$ sequence for $ I_{f,g}$ at level $\ba$ if $ I_{f,g}(u_n,v_n)\to \ba$ and $I_{f,g}'(u_n,v_n)\to 0$ in $\big(\dot{H}^s(\Rn)\times\Hs\big)'$. It is easy to see that the weak limit of a $(PS)$ sequence of $I_{f,g}$ solves \eqref{MAT1} except the positivity.

However the main
difficulty is that the $(PS)$ sequence may not converge strongly and hence the weak limit can be zero even if $\ba>0.$ The main purpose of this section is to classify $(PS)$ sequences for
the functional $I_{f,g}$.

\begin{proposition}\label{PSP}
				
Let $\{(u_n,v_n)\} \subset \dot{H}^s(\R^N)\times\Hs$ be a $(PS)$ sequence for $I_{f,g}$ at a level $\ga$. Then there exists a subsequence $($still denoted by $\{(u_n,v_n)\})$ for which
the following hold :\\
there exist an integer $k\geq 0$, sequences $\{x_{n}^{i}\}_{n}\subset\Rn$, $r_n^i>0$ for $1\leq i \leq k$, pair of functions $(u,v),\;(\tilde{u}_i,\tilde{v}_i)\in\Hs\times\Hs$ for $1\leq i\leq k$ such that $(u,v)$ satisfies \eqref{MAT1} without the signed restrictions and
\begin{equation}\label{PSD2}
		\left\{\begin{aligned}
			&(-\Delta)^s \tilde{u}_i = \frac{\al}{2_s^*}|\tilde{u}_i|^{\al-2}\tilde{u}_i|\tilde{v}_i|^{\ba}\;\;\text{in}\;\mathbb{R}^{N},\\
			&(-\Delta)^s \tilde{v}_i = \frac{\ba}{2_s^*}|\tilde{v}_i|^{\ba-2}\tilde{v}_i|\tilde{u}_i|^{\al}\;\;\text{in}\;\mathbb{R}^{N},\\
		\end{aligned}
		\right.
	\end{equation}
	\be
		\begin{split}
			(u_n,v_n)=(u,v) +\sum_{i=1}^{k}(\tilde u_i,\tilde v_i)^{r_n^i, x_n^i}+o(1),\\
						\mbox{where } (\tilde u_i,\tilde v_i)^{r, y}:=r^{-\frac{N-2s}{2}}\left(\tilde u_i(\tfrac{x-y}{r}),\tilde v_i(\tfrac{x-y}{r})\right) \lab{12-13-3}\\
							\mbox{and }\,\,  o(1)\to 0\quad \text{in}\,\,\Hs\times\Hs,\\
\ga= I_{f,g}(u,v)+\sum_{i=1}^{k}I_{0,0}(\tilde u_i,\tilde v_i)+o(1),
		\end{split}
		\ee		

		\be\begin{split}
		r_n^i\to 0 \,\,
		\mbox{and either $x_n^i\to x^i\in\Rn$ or }|x_n^i|\to\infty,\quad  1\leq i\leq k,\\		
		 \bigg|\log\bigg(\frac{r^i_n}{r_n^{j}}\bigg)\bigg|+\bigg|\frac{x_n^i-x_n^j}{r_n^i}\bigg|
\To\infty\quad \mbox{for } i\neq j,\, \, 1\leq i,\, j\leq k, \end{split}\ee
				
		\noi where  in the case $k=0$ the above expressions hold without $(\tilde u_i,\tilde v_i)$, $x_{n}^{i}$ and $r_n^i$.
				\end{proposition}

\begin{remark}\lab{r:psp}
{\rm From Proposition \ref{PSP}, we see that if $\{(u_n,v_n)\} $ is any nonnegative $(PS)$ sequence for $I_{f,g}$ at level $\ga$, then $\{(u_n,v_n)\} $ satisfies the $(PS)$ condition  if $\ga$ can not be decomposed as $\ga=I_{f,g}(u,v)+\sum_{i=1}^{k}I_{0,0}(\tilde u_i, \tilde v_i)$, where $k\geq 1$ and $(\tilde u_i,\tilde v_i)$ is a solution of \eqref{PSD2}.}
\end{remark}

\medskip

Before starting the proof of this proposition, we prove some lemmas which will be used in proving Proposition~\ref{PSP}.
\begin{lemma}\lab{BL}\cite[Theorem 2]{BL}
Let $j:\mathbb{C}\to\mathbb{C}$ be a continuous function with $j(0)=0$ and satisfy the following hypothesis that for every $\eps>0$, there exists two continuous functions $\varphi_\eps$ and $\psi_\eps$ such that
$$|j(\tilde a+\tilde b)-j(\tilde a)|\leq \eps\varphi_\eps(\tilde a)+\psi_{\eps}(\tilde b) \quad\forall\, \tilde a,\,\tilde b\in \mathbb{C}.$$ Further, let $f_n=f+g_n$ be a sequence of measurable functions from $\Rn$ to $\mathbb{C}$ such that

(i) $g_n\to 0$ a.e.

(ii) $j(f)\in L^1(\Rn)$.

(iii) $\int_{\Rn}\varphi_\eps\big(g_n(x)\big){\rm d}x\leq C<\infty$, for some constant $C$, independent of $\eps$ and $n$.

(iv) $\int_{\Rn}\psi_\eps(f(x)){\rm d}x<\infty$ for all $\eps>0$.

\noindent
Then $$\int_{\Rn}|j(f+g_n)-j(f)-j(g_n)|{\rm d}x\To 0, \quad\text{as}\quad n\to\infty.$$
\end{lemma}

\begin{lemma}\lab{l:j}
Let $\al,\ba>1$. Then for every $\eps>0$, there exists $C_\eps>0$ such that
$$\big||x+a|^\al|y+b|^\ba-|x|^\al|y|^\ba\big|\leq \eps(|x|^{\al+\ba}+|y|^{\al+\ba})+C_\eps(|a|^{\al+\ba}+|b|^{\al+\ba})$$
holds for all $x,\, y,\, a,\, b \in \R$.
\end{lemma}
\begin{proof}
Let $\eps>0$ be arbitrary. Then there exists $C_\eps>0$ such that
\Bea
|x+a|^\al|y+b|^\ba-|x|^\al|y|^\ba &=&|y+b|^\ba(|x+a|^\al-|x|^\al)+|x|^\al(|y+b|^\ba-|y|^\ba)\\
&\leq&2^{\ba-1}( |y|^\ba+|b|^\ba)\bigg(\frac{\eps/2}{2^\ba-1}|x|^\al+C_\eps|a|^\al\bigg)+|x|^\al\bigg(\frac{\eps}{2}|y|^\ba+C_\eps|b|^\ba\bigg)\\
&\leq&\eps\bigg(|x|^\al|y|^\ba+\frac{1}{2}|b|^\ba|x|^\al\bigg)+C_\eps(|x|^\al|b|^\ba+|y|^\ba|a|^\al+|a|^\al|b|^\ba)\\
&\leq& \eps(|x|^{\al+\ba}+|y|^{\al+\ba})+C'_\eps(|a|^{\al+\ba}+|b|^{\al+\ba}),
\Eea
where in the last inequality we have used Young's inequality with  different $\eps$. This completes the proof.
\end{proof}

\begin{lemma}\lab{l:BL} If $u_n\deb u$ and $v_n\deb v$ in $\dot{H}^s(\Rn)$. Then
$$\int_{\Rn}\bigg(|u_n|^{\alpha}|v_n|^\beta-|u|^{\alpha}|v|^\beta-|u_n-u|^{\alpha}|v_n-v|^\beta\bigg){\rm d}x=o(1).$$
\end{lemma}

\begin{proof}
Define $j:\R^2\to\R$ defined by $j(x,y)=|x|^\alpha|y|^\beta$. Then $j$ satisfies the  hypothesis of Lemma~\ref{BL}. Next considering $$f_n:=(u_n, v_n), \quad f=(u, v), \quad g_n=(u_n-u, v_n-v),$$
we see that all the hypothesis of Lemma~\ref{BL} are satisfied. Hence the lemma follows from Lemma~\ref{BL}.
\end{proof}

\begin{lemma}\label{L2-2}
			Let $\{(u_n,v_n)\}$ weakly converge to $(u,v)$ in $\Hs\times\Hs$ and pointwise a.e. in $\R^N\times\R^N$, then
			\bea
			\int_{\Rn}|u_n|^{\al-2}u_n|v_n|^{\ba}\phi\;{\rm d}x &\longrightarrow& \int_{\Rn}|u|^{\al-2}u|v|^{\ba}\phi\;{\rm d}x \quad \text{as}\;\; n\to\infty,\label{PDP3}\\
			\text{and }\int_{\Rn}|u_n|^{\al}|v_n|^{\ba-2}v_n\psi\;{\rm d}x &\longrightarrow& \int_{\Rn}|u|^{\al}|v|^{\ba-2}v\psi\;{\rm d}x \quad \text{as}\;\; n\to\infty,\label{PDP4}
			\eea
			for all $(\phi,\psi)\in\Hs\times\Hs.$
		\end{lemma}
		\begin{proof}
Set $$M:=\max\bigg\{\|u_n\|_{L^{2^*_s}(\Rn)}^{\al-1}, \, \|v_n\|_{L^{2^*_s}(\Rn)}^{\ba}\, \|u\|_{L^{2^*_s}(\Rn)}^{\al-1}\, \|v\|_{L^{2^*_s}(\Rn)}^{\ba}\bigg\}.$$
Using the Sobolev inequality we see that $M$ is well-defined. Let $\phi\in \dot{H}^s(\Rn)$ and $\eps>0$ be arbitrary.  Then, there exists $R=R(\eps)>0$ such that $\big(\int_{B(0,R)^c}|\phi|^{2^*_s}{\rm d}x\big)^\frac{1}{2^*_s}<\frac{\eps}{2M^2}$.
Note that,				
$$\int_{\Rn}\big(|u_n|^{\al-2}u_n|v_n|^{\ba}-|u|^{\al-2}u|v|^{\ba}\big)\phi\;{\rm d}x=\bigg(\int_{B(0,R)} +\int_{B(0,R)^c}\bigg)\big(|u_n|^{\al-2}u_n|v_n|^{\ba}-|u|^{\al-2}u|v|^{\ba}\big)\phi$$
and using H\"{o}lder inequality
\Bea
&&\int_{B(0,R)^c}\big(|u_n|^{\al-2}u_n|v_n|^{\ba}-|u|^{\al-2}u|v|^{\ba}\big)\phi\;{\rm d}x\\
 &\leq& \bigg(\int_{\Rn}|u_n|^{2^*_s}{\rm d}x\bigg)^{(\al-1)/2^*_s} \bigg(\int_{\Rn}|v_n|^{2^*_s}{\rm d}x\bigg)^{\ba/2^*_s} \bigg(\int_{B(0,R)^c}|\phi|^{2^*_s}{\rm d}x\bigg)^{1/2^*_s}\\
&&\quad+\bigg(\int_{\Rn}|u|^{2^*_s}{\rm d}x\bigg)^{(\al-1)/2^*_s} \bigg(\int_{\Rn}|v|^{2^*_s}{\rm d}x\bigg)^{\ba/2^*_s} \bigg(\int_{B(0,R)^c}|\phi|^{2^*_s}{\rm d}x\bigg)^{1/2^*_s}\\
 &<&\eps.
\Eea			
On the other hand, using H\"{o}lder inequality as above, it is also easily checked that $\displaystyle \big(|u_n|^{\al-2}u_n|v_n|^{\ba}-|u|^{\al-2}u|v|^{\ba}\big)\phi$ is equi-integrable in $B(0,R)$. Therefore, applying Vitaly's convergence theorem it follows that
 $$\lim_{n\to\infty}\int_{B(0,R)}\big(|u_n|^{\al-2}u_n|v_n|^{\ba}-|u|^{\al-2}u|v|^{\ba}\big)\phi\;{\rm d}x=0.$$
Hence the lemma follows.					
\end{proof}	

{\bf Proof of Proposition \ref{PSP}:}

\begin{proof}
      We divide the proof into few steps.\\

      \underline{\bf Step 1:} Using standard arguments it follows that $(PS)$ sequences for $I_{f,g}$ are bounded in $\Hs\times \Hs$. More precisely, as $n\to\infty$
      \begin{eqnarray*}
\ga+o(1)+o(1)\|(u_n,v_n)\|_{\dot{H}^s\times \dot{H}^s}&\geq& I_{f,g}(u_n,v_n) \, - \,
      \frac{1}{2_s^*} \prescript{}{(\dot{H}^s\times\dot{H}^s)'}{\langle} I'_{f,g}(u_n,v_n), (u_n,v_n){\rangle}_{\dot{H}^s\times \dot{H}^s} \\
     &\geq&\bigg(\frac{1}{2}-\frac{1}{2_s^*}\bigg)\|(u_n,v_n)\|_{\dot{H}^s\times\dot{H}^s}^{2}\\
       &&\qquad - \left(1- \frac{1}{2_s^*} \right)\left(\|f\|_{\hms}\|u_n\|_{\dot{H}^s}+\|g\|_{\hms}\|v_n\|_{\dot{H}^s}\right)\\
       &\geq&\bigg(\frac{1}{2}-\frac{1}{2_s^*}\bigg)\|(u_n,v_n)\|_{\dot{H}^s\times\dot{H}^s}^{2}\\
       &&\qquad-  \left(1- \frac{1}{2_s^*} \right)(\|f\|_{\hms}+\|g\|_{\hms})\|(u_n,v_n)\|_{\dot{H}^s\times \dot{H}^s}.
      \end{eqnarray*}
    This immediately implies that $\{(u_n,v_n)\}$ is bounded in $\Hs\times\Hs$. Consequently, up to a subsequence, $ (u_n,v_n) \rightharpoonup (u,v)$ in $\Hs\times\Hs$.  Further,  $\prescript{}{\hhms}{\langle}I_{f,g}'(u_n,v_n), (\phi,\psi){\rangle}_{\dot{H}^s\times\dot{H}^s}\rightarrow 0 $ implies
\Bea
&&\big\langle (u_n,v_n), (\phi,\psi)\big\rangle_{\dot{H}^s\times\dot{H}^s}- \frac{\al}{2^*_s}\int_{\R{^N}}|u_n|^{\al-2}u_n|v_n|^{\ba}\phi\,{\rm d}x-\frac{\ba}{2^*_s}\int_{\R{^N}}|v_n|^{\ba-2}v_n|u_n|^{\al}\psi\,{\rm d}x\\
&&\qquad\qquad\qquad-\prescript{}{(\dot{H}^s)'}{\langle}f,\phi{\rangle}_{\dot{H}^s} -\prescript{}{(\dot{H}^s)'}{\langle}g,\psi{\rangle}_{\dot{H}^s}=o(1).\Eea
Passing to the limit using Lemma~\ref{L2-2}, we see that  $(u,v)$ satisfies \eqref{MAT1} without signed restrictions.

\medskip

	\underline{\bf Step 2:} In this step we show that $\{(u_n-u,v_n-v)\}$ is a $(PS)$ sequence for $I_{0,0}$ at the level
$\ga-I_{f,g}(u,v)$

\medskip

To see this, first we observe that as $n\to\infty$
\be
\|u_n-u\|_{\dot{H}^s}=\|u_n\|_{\dot{H}^s}^2-\|u\|_{\dot{H}^s}^2+o(1), \quad
\|v_n-v\|_{\dot{H}^s}^2 = \|v_n\|_{\dot{H}^s}^2-\|v\|_{\dot{H}^s}^2+o(1).\no
\ee
Using this along with the fact that $(u_n,v_n)\rightharpoonup (u,v)$, $f,\;g\in (\dot{H}^s(\Rn))'$ and Lemma~\ref{l:BL} yields
  \Bea
      I_{0,0}(u_n-u,v_n-v)
     &=& \frac{1}{2}\|(u_n-u,v_n-v)\|_{\phs}^2 -\frac{1}{2_s^*}\int_{\Rn}|u_n-u|^{\al}|v_n-v|^{\ba}{\rm d}x\no\\
     &=&\frac{1}{2}(\|u_n\|_{\dot{H}^s}^2-\|u\|_{\dot{H}^s}^2)+\frac{1}{2}(\|v_n\|_{\dot{H}^s}^2-\|v\|_{\dot{H}^s}^2)
     -\frac{1}{2_s^*}\int_{\Rn}|u_n|^{\al}|v_n|^{\ba}{\rm d}x\\
     &&+\frac{1}{2_s^*}\int_{\Rn}|u_n|^{\al}|v_n|^{\ba}{\rm d}x
    -\prescript{}{\hms}{\langle}f, u_n{\rangle}_{\dot{H}^s}-\prescript{}{\hms}{\langle}g, v_n{\rangle}_{\dot{H}^s}\\
    &&+\prescript{}{\hms}{\langle}f, u{\rangle}_{\dot{H}^s}+
     \prescript{}{\hms}{\langle}g, v{\rangle}_{\dot{H}^s}\\
     &&+\frac{1}{2_s^*}\int_{\Rn}\left\{|u_n|^{\al}|v_n|^{\ba}-|u|^{\al}|v|^{\ba}-|u_n-u|^{\al}|v_n-v|^{\ba}\right\}\;{\rm d}x+o(1)\no\\
     &=&I_{f,g}(u_n,v_n)-I_{f,g}(u,v)+o(1).
     \Eea
  Next,  as $(u_n-u, v_n-v)\deb (0,0)$ in $\Hs\times\Hs$, applying Lemma~\ref{L2-2}, we obtain
\bea
     &\;&\prescript{}{\hhms}{\langle} I'_{0,0}(u_n-u,v_n-v),\, (\phi,\psi){\rangle}_{\phs}\no\\
      &=& \langle (u_n-u,v_n-v),(\phi,\psi)\rangle_{\phs}-\frac{\al}{2_s^*}\int_{\Rn}|u_n-u|^{\al-2}(u_n-u)|v_n-v|^{\ba}\phi\, {\rm d}x\no\\ &\;&\qquad\qquad-\frac{\ba}{2_s^*}\int_{\Rn}|u_n-u|^{\al}|v_n-v|^{\ba-2}(v_n-v)\psi\, {\rm d}x\no\\
       &=&o(1).
  \eea
  This completes Step 2.

 \medskip

\underline{\bf Step 3:}  Rescaling of $\{(u_n,v_n)\}_n$ in the nontrivial case.

If $(u_n,v_n)\to (u,v) \mbox{ in } \Hs\times\Hs$, then the theorem is proved with $k=0$. Therefore, we assume $(u_n,v_n)\not\to (u,v) \mbox{ in } \Hs\times\Hs$. Set, $$\tilde u_n:=u_n-u, \quad \tilde v_n:=v_n-v. $$
Therefore, we are in the case where $(\tilde u_n, \tilde v_n)\not\To(0,0)$ in $\Hs\times\Hs$.  Since, by Step 2, $\{(\tilde u_n, \tilde v_n)\}$ is a bounded $(PS)$ sequence for $I_{0,0}$, we have $I'_{0,0}(\tilde u_n, \tilde v_n)(\tilde u_n, \tilde v_n)=o(1)$. Therefore, up to a subsequence
$$0<\|(\tilde u_n, \tilde v_n)\|^2_{\phs}=\int_{\Rn}|\tilde u_n|^\al |\tilde v_n|^\ba{\rm d}x\leq\int_{\Rn}|\tilde u_n|^{2^*_s}{\rm d}x+\int_{\Rn}|\tilde v_n|^{2^*_s}{\rm d}x\leq\|(\tilde u_n, \tilde v_n)\|_{L^{2^*_s}\times L^{2^*_s}}^{2^*_s}.
$$
Thus $(\tilde u_n, \tilde v_n)\not\To(0,0)$ in $L^{2^*_s}(\Rn)\times L^{2^*_s}(\Rn)$.
Consequently, $$\inf_{n}\|(\tilde u_n, \tilde v_n)\|_{L^{2^*_s}\times L^{2^*_s}}\geq \de\quad\text{for some}\quad \de>0.$$
Hence, applying Lemma~\ref{l:12-13-1},
$$\de\leq C\|(\tilde u_n,\tilde v_n)\|_{\dot{H}^s\times \dot{H}^s}^\theta\|(\tilde u_n,\tilde v_n)\|^{1-\theta}_{L^{2, (N-2s)}\times L^{2, (N-2s)}}\leq C'\|(\tilde u_n,\tilde v_n)\|^{1-\theta}_{L^{2, (N-2s)}\times L^{2, (N-2s)}},$$
that is, $$\|(\tilde u_n,\tilde v_n)\|_{L^{2, (N-2s)}\times L^{2, (N-2s)}}\geq C_1 \quad\text{for some }\,\, C_1>0.$$
Comparing the above inequality with \eqref{3-9-1} yields existence of some $\bar C>0$ such that
$$\bar C\leq\|(\tilde u_n,\tilde v_n)\|_{L^{2, (N-2s)}\times L^{2, (N-2s)}}^2\leq \bar C^{-1},$$
that is
$$\bar C\leq\sup_{x\in\Rn,\, R>0} R^{N-2s}\fint_{B(x,R)}\big(|\tilde u_n|^2+|\tilde v_n|^2\big){\rm d}y\leq \bar C^{-1}.$$
As a result, for every $n\in \N$, there exists $x_n\in\Rn$ and $r_n>0$ such that
\be\lab{3-9-3}
r_n^{N-2s}\fint_{B(x_n,r_n)}\big(|\tilde u_n|^2+|\tilde v_n|^2\big){\rm d}y\geq \|(\tilde u_n,\tilde v_n)\|_{L^{2, (N-2s)}\times L^{2, (N-2s)}}-\frac{\bar C}{2n}\geq \frac{\bar C}{2}>0.
\ee
Now define
$$\tilde{\tilde{u}}_n := r_n^{\frac{N-2s}{2}}\tilde u_n(r_nx+x_n), \quad
\tilde{\tilde{v}}_n := r_n^{\frac{N-2s}{2}}\tilde v_n(r_nx+x_n).$$
In view of the scaling invariance of the $\Hs$ norm and $L^{2^*_s}(\Rn)$ norm, $\{(\tilde{\tilde{u}}_n, \tilde{\tilde{v}}_n )\}$ is a bounded sequence in $\Hs\times\Hs$ and  up to a subsequence
$$(\tilde{\tilde{u}}_n, \tilde{\tilde{v}}_n )\deb (\tilde u, \tilde v) \quad\text{in}\, \Hs\times\Hs\quad\text{and}\quad (\tilde{\tilde{u}}_n, \tilde{\tilde{v}}_n )\To (\tilde u, \tilde v) \quad\text{in}\, L^2_{loc}(\Rn)\times L^2_{loc}(\Rn).$$
Therefore, using change of variable, we observe from \eqref{3-9-3} that

$$0<r_n^{-2s}\int_{B(x_n,r_n)}\big(|\tilde u_n|^2+|\tilde v_n|^2\big){\rm d}y=\int_{B(0,1)}
\big(|\tilde{\tilde{u}}_n(z)|^2+|\tilde{\tilde{v}}_n(z)|^2\big){\rm d}z
\To\int_{B(0,1)}(|\tilde u|^2+|\tilde v|^2){\rm d}x.$$
Hence $(\tilde u, \tilde v)\neq (0,0)$. Clearly, up to a subsequence, either $x_n\to x_0\in\Rn$ or $|x_n|\to\infty$. Further, as
$(\tilde{\tilde{u}}_n, \tilde{\tilde{v}}_n )\deb (\tilde u, \tilde v)\neq (0,0)$ in $\Hs\times\Hs$ and $({\tilde{u}}_n, {\tilde{v}}_n )\deb (0,0)$ in $\Hs\times\Hs$, we infer that
 $r_n\to 0$.

\medskip

\underline{\bf Step 4:} In this step we prove that $(\tilde u,\tilde v)$ solves
\begin{equation}\label{PSD5}
	\left\{\begin{aligned}
		&(-\Delta)^s\tilde u = \frac{\al}{2_s^*}|\tilde u|^{\al-2}\tilde{u}|\tilde{v}|^{\ba}\,\, \text{ in }\,\Rn,\\
		&(-\Delta)^s\tilde{v} = \frac{\ba}{2_s^*}|\tilde{u}|^{\al}|\tilde{v}|^{\ba-2}\tilde{v}\,\,\text{ in }\, \Rn.\\
	\end{aligned}
	\right.
\end{equation}
To this aim, it is enough to show that for arbitrary $(\va,\psi)\in C_c^\infty(\Rn)\times C_c^\infty(\Rn)$ it holds
$$\langle \tilde u,\va\rangle_{\dot{H}^s}+\langle \tilde v,\psi\rangle_{\dot{H}^s}=\frac{\al}{2^*_s}\int_{\Rn}|\tilde u|^{\al-2}\tilde{u}|\tilde{v}|^{\ba}\va\,{\rm d}x+\frac{\ba}{2_s^*}\int_{\Rn}|\tilde{u}|^{\al}|\tilde{v}|^{\ba-2}\tilde{v}\psi\,{\rm d}x.$$
Let $\varphi, \psi\in C^\infty_c(\Rn)$ be arbitrary. As $(\tilde{\tilde{u}}_n, \tilde{\tilde{v}}_n )\deb (\tilde u, \tilde v)$ in $\Hs\times\Hs$, using change of variables and Step 2, that is $\{(\tilde u_n, \tilde v_n)\}$ is a $(PS)$ sequence for $I_{0,0}$, we deduce

\begin{align}
	\langle\tilde u,\, \varphi \rangle_{\dot{H}^s}+
	\langle\tilde v,\, \psi \rangle_{\dot{H}^s} &=\lim_{n\to\infty}\big(\langle\tilde{\tilde{u}}_n, \va\rangle_{\dot{H}^s}+
	\langle\tilde{\tilde{v}}_n,\, \psi \rangle_{\dot{H}^s}\big)\no\\
	&=\lim_{n\to\infty}\iint_{\R^{2N}}\frac{r_n^\frac{N-2s}{2}\big(\tilde u_n(r_nx+x_n)-\tilde u_n(r_ny+x_n)\big)
\big(\varphi(x)-\varphi(y)\big)}{|x-y|^{N+2s}}{\rm d}x{\rm d}y\no\\
&\quad+\lim_{n\to\infty}\iint_{\R^{2N}}\frac{r_n^\frac{N-2s}{2}\big(\tilde v_n(r_nx+x_n)-\tilde v_n(r_ny+x_n)\big)
\big(\psi(x)-\psi(y)\big)}{|x-y|^{N+2s}}{\rm d}x{\rm d}y\no\\
&=\lim_{n\to\infty}\iint_{\R^{2N}}\frac{r_n^{-\frac{N-2s}{2}}\big(\tilde u_n(x)-\tilde u_n(y)\big)\big(\varphi\big(\dfrac{x-x_n}{r_n}\big)
-\varphi\big(\dfrac{y-x_n}{r_n}\big)\big)}{|x-y|^{N+2s}}{\rm d}x{\rm d}y\no\\
&\quad+\lim_{n\to\infty}\iint_{\R^{2N}}\frac{r_n^{-\frac{N-2s}{2}}\big(\tilde v_n(x)-\tilde v_n(y)\big)\big(\psi\big(\dfrac{x-x_n}{r_n}\big)
-\psi\big(\dfrac{y-x_n}{r_n}\big)\big)}{|x-y|^{N+2s}}{\rm d}x{\rm d}y\no\\
&=\lim_{n\to\infty}\bigg[\frac{\al}{2^*_s}\int_{\Rn}|\tilde u_n|^{\al-2}\tilde{u_n}|\tilde{v_n}|^{\ba}\tilde \va_n{\rm d}x+\frac{\ba}{2_s^*}\int_{\Rn}|\tilde{u_n}|^{\al}|\tilde{v_n}|^{\ba-2}\tilde{v_n}\tilde \psi_n{\rm d}x\bigg], \lab{3-9-5}
\end{align}
where $$\tilde \va_n(x):=r_n^{-\frac{N-2s}{2}}\va\big(\dfrac{x-x_n}{r_n}\big)\quad\text{and}\quad \tilde \psi_n(x):=r_n^{-\frac{N-2s}{2}}\psi\big(\dfrac{x-x_n}{r_n}\big).$$
Again applying change of variable to \eqref{3-9-5} yields us
\Bea
\text{RHS of \eqref{3-9-5}}&=&\lim_{n\to\infty}\bigg[\frac{\al}{2^*_s}\int_{\Rn}|\tilde{\tilde{u}}_n|^{\al-2}\tilde{\tilde{u}}_n|\tilde{\tilde{v}}_n|^{\ba}\va\,{\rm d}x+\frac{\ba}{2_s^*}\int_{\Rn}|\tilde{\tilde{u}}_n|^{\al}|\tilde{\tilde{v}}_n|^{\ba-2}\tilde{\tilde{v}}_n\psi\,{\rm d}x\bigg]\\
&=&\frac{\al}{2^*_s}\int_{\Rn}|\tilde u|^{\al-2}\tilde{u}|\tilde{v}|^{\ba}\va\,{\rm d}x+\frac{\ba}{2_s^*}\int_{\Rn}|\tilde{u}|^{\al}|\tilde{v}|^{\ba-2}\tilde{v}\psi\,{\rm d}x,
\Eea
where the last equality is obtained by Lemma~\ref{L2-2}. This completes Step 4.

\medskip

Now define,
\be\lab{4-9-1}w_n(x) := \tilde u_n(x)-r_n^{-\frac{N-2s}{2}}\tilde{u}\left(\frac{x-x_n}{r_n}\right) \quad\text{and}\quad z_n(x) := \tilde v_n(x)-r_n^{-\frac{N-2s}{2}}\tilde{v}\left(\frac{x-x_n}{r_n}\right).\ee

\medskip

\underline{\bf Step 5:}
In this step we show that $\{(w_n,z_n)\}$ is a $(PS)$ sequence for $ I_{0,0}$ at the level $\ga-I_{f,g}(u,v)-I_{0,0}(\tilde u,\tilde v)$.

To prove that, first we set
\be\lab{4-9-2}\tilde{w}_n := r_n^{\frac{N-2s}{2}}w_n(r_nx+x_n), \quad\text{and}\quad
\tilde{z}_n := r_n^{\frac{N-2s}{2}}z_n(r_nx+x_n).\ee
Combining \eqref{4-9-1} and \eqref{4-9-2} yields
 $$\tilde{w}_n=\tilde{\tilde{u}}_n-\tilde u, \quad \tilde{z}_n=\tilde{\tilde{v}}_n-\tilde v,$$
and from the scaling invariance in the norm of $\Hs\times\Hs$ gives
$$\|(w_n, z_n)\|_{\phs}=\|(\tilde w_n, \tilde z_n)\|_{\phs}=\|(\tilde{\tilde{u}}_n-\tilde u, \tilde{\tilde{v}}_n-\tilde{v})\|_{\phs}.$$
A straight forward computation using  the above equality, change of variables and Lemma~\ref{l:BL} yields
\begin{align*}
I_{0,0}(w_n,z_n)&=\frac{1}{2}\|\tilde{\tilde{u}}_n-\tilde u\|^2_{\dot{H}^{s}}
+\frac{1}{2}\|\tilde{\tilde{v}}_n-\tilde{v}\|^2_{\dot{H}^{s}}
-\frac{1}{2^*_s}\int_{\Rn}|w_n|^{\al}|z_n|^{\ba}{\rm d}x\\
&=\frac{1}{2}\big(\|\tilde{\tilde{u}}_n\|^2_{\dot{H}^{s}}-\|\tilde u\|^2_{\dot{H}^{s}}+\|\tilde{\tilde{v}}_n\|^2_{\dot{H}^{s}}-\|\tilde v\|^2_{\dot{H}^{s}}+o(1)\big)-\frac{1}{2^*_s}\int_{\Rn}|\tilde{\tilde{u}}_n-\tilde u|^\al|\tilde{\tilde{v}}_n-\tilde v|^\ba{\rm d}x\\
&=\frac{1}{2}\|(\tilde{\tilde{u}}_n, \tilde{\tilde{v}}_n)\|_{\phs}^2-\frac{1}{2}\|(\tilde u, \tilde v)\|_{\phs}^2-\frac{1}{2^*_s}\bigg[\int_{\Rn} |\tilde{\tilde{u}}_n|^\al|\tilde{\tilde{v}}_n|^\ba{\rm d}x -\int_{\Rn}|\tilde{u}|^\al|\tilde v|^\ba{\rm d}x\bigg]+o(1)\\
&=I_{0,0}(\tilde{\tilde{u}}_n, \tilde{\tilde{v}}_n)-I_{0,0}(\tilde u, \tilde v)+o(1)\\
&=I_{0,0}(\tilde u_n, \tilde v_n)-I_{0,0}(\tilde u, \tilde v)+o(1)\\
&=\ga-I_{f,g}(u,v)-I_{0,0}(\tilde u, \tilde v)+o(1),
\end{align*}
where in the last equality we have used Step~2. Now, to complete the proof of Step~5, we left to show that $\langle I'_{0,0}(w_n,z_n)(\va,\psi)\rangle=0$ for all $(\va,\psi)\in C^\infty_c(\Rn)\times C^\infty_c(\Rn)$. Let $(\va,\psi)\in C^\infty_c(\Rn)\times C^\infty_c(\Rn)$ be arbitrary and set
$${\va}_n := r_n^{\frac{N-2s}{2}}\va(r_nx+x_n), \quad
{\psi}_n := r_n^{\frac{N-2s}{2}}\psi(r_nx+x_n).$$
Thus $\va_n\deb 0$ and $\psi_n\deb 0$ in $\Hs$ as $r_n\to 0$.
Observe that applying change of variables,
\begin{align*}
\big\langle (w_n,z_n), (\va,\psi)\big\rangle_{\dot{H}^s\times\dot{H}^s}&=\langle w_n,\va\rangle_{\dot{H}^s}+ \langle z_n,\psi\rangle_{\dot{H}^s}\\
&=\langle \tilde w_n,\va_n\rangle_{\dot{H}^s}+ \langle \tilde z_n,\psi_n\rangle_{\dot{H}^s}\\
&=\langle\tilde{\tilde{u}}_n-\tilde u, \va_n\rangle_{\dot{H}^s}+\langle\tilde{\tilde{v}}_n-\tilde v, \psi_n\rangle_{\dot{H}^s}.
\end{align*}
Therefore,
\begin{align*}
\langle I'_{0,0}(w_n,z_n)(\va,\psi)\rangle&=\langle\tilde{\tilde{u}}_n-\tilde u, \va_n\rangle_{\dot{H}^s}+\langle\tilde{\tilde{v}}_n-\tilde v, \psi_n\rangle_{\dot{H}^s}\\
&\qquad-\frac{\al}{2^*_s}\int_{\Rn}|\tilde{\tilde{u}}_n-\tilde u|^{\al-2}(\tilde{\tilde{u}}_n-\tilde u)|\tilde{\tilde{v}}_n-\tilde v|^\ba\va_n{\rm d}x\\
&\qquad-\frac{\ba}{2^*_s}\int_{\Rn}|\tilde{\tilde{u}}_n-\tilde u|^{\al}|\tilde{\tilde{v}}_n-\tilde v|^{\ba-2}(\tilde{\tilde{v}}_n-\tilde u)\psi_n{\rm d}x\\
&=o(1),
\end{align*}
where the last equality follows by change of variable and an argument similar to Step~2.  This concludes Step 5.

Now, starting from a $(PS)$ sequence $\{(\tilde u_n, \tilde v_n)\}$ for $I_{0,0}$ we have extracted another $(PS)$ sequence $\{(w_n, z_n)\}$ at a level which is strictly lower than the previous one, with a fixed minimum amount of decrease (as it is easy to check that $\displaystyle I_{0,0}(\tilde u,\tilde v)\geq \frac{s}{N}S_{\al,\ba}^\frac{N}{2s}$
). On the other hand, as $\sup_n\|(\tilde u_n, \tilde v_n)\|_{\phs}\leq C$ (finite),  this process should terminate after finitely many steps and the last $(PS)$ sequence strongly converges to $0$. Further, $\bigg|\log\big(\frac{r^i_n}{r_n^j}\big)\bigg|+\bigg|\frac{x_n^i-x_n^j}{r_n^i}\bigg|\To\infty\quad \mbox{for } i\neq j,\, \, 1\leq i,\, j\leq m$ (see
\cite[Theorem 1.2]{PS-2}).
This completes the proof.

\end{proof}

\section{Multiplicity in the nonhomogeneous case}

In this section we aim to prove Theorem~\ref{th:ex-f}. For that first we would like to establish existence of two positive critical points of the functional
\be\label{S3.1}
J_{f,g}(u,v)=\frac{1}{2}\|(u,v)\|_{\phs}^2-\frac{1}{2_s^*}\int_{\Rn}u_+^{\al}v_+^{\ba}\;{\rm d}x-\prescript{}{(\dot{H}^s)'}\langle f,u\rangle_{\dot{H}^s}-\prescript{}{(\dot{H}^s)'}\langle g,v\rangle_{\dot{H}^s}
\ee
where $f,\,g$ are nontrivial nonneagtive functionals on $(\Hs)'$ with $\ker(f)=\ker(g)$.

\begin{remark}\label{SR31} {\rm
If $(u,v)\in\Hs\times\Hs$ is a nontrivial critical point of $J_{f,g}$ then $(u,v)$ solves
\begin{equation}\label{S32}
\left\{\begin{aligned}
	&(-\Delta)^s u = \frac{\al}{2_s^*} u_+^{\al-1}v_+^{\ba}+f(x) \quad\text{ in }\Rn,\\
	&(-\Delta)^sv = \frac{\ba}{2_s^*}u_+^{\al}v_+^{\ba-1}+g(x)\quad\text{ in } \Rn.\\
\end{aligned}
\right.
\end{equation}
Note that taking $(\phi,\psi)=(u_-,v_-)$ as a test function in \eqref{S32}, we obtain
	\begin{align}
	-\|(u_-,v_-)\|^2_{\phs} \, &-\, \iint_{\R^{2N}} \frac{[u_+(y)u_-(x)+u_+(x)u_-(y)]}{|x - y|^{N + 2s}} \, {\rm d}x \, {\rm d}y\no\\
	&-\, \iint_{\R^{2N}} \frac{[v_+(y)v_-(x)+v_+(x)v_-(y)]}{|x - y|^{N + 2s}} \, {\rm d}x \, {\rm d}y\no\\
	&=\prescript{}{(\dot{H}^s)'}{\langle}f,u_-{\rangle}_{H^s}+\prescript{}{(\dot{H}^s)'}{\langle}g,v_-{\rangle}_{H^s}\geq 0.\no
\end{align}
	which in turn implies $u_-=0$ and $v_-=0.$ Therefore $u, v\geq 0$ and $(u,v)$ is a solution of \eqref{MAT1} without strict positivity condition.

Next, we assert that $(u,v)\neq (0,0)$ implies $u\neq 0$ and $v\neq 0.$ Suppose not, that is assume for instance that $u\neq 0$ but $v=0$. Then taking $(\phi,\psi)=(u,0)$ as test function we get $\|u\|_{\Hs}^2=\prescript{}{(\dot{H}^s)'}\langle f,u\rangle_{\dot{H}^s}.$
Further choosing $(\phi,\psi)=(0,u)$ as test function, we have $\prescript{}{(\dot{H}^s)'}\langle g,u\rangle_{\dot{H}^s}=0$. These together with the hypothesis that ker$(f)$=ker$(g)$ implies $\|u\|_{\dot{H}^s}=0.$ This contradicts $(u,v)\neq (0,0).$ Similarly we can show that if $u=0$ then $v=0$ too. Hence our assertion follows. Next, we claim that $u>0,$ and $v>0$ in $\Rn$. Taking $(\phi,0)$ as test function where $\phi\geq 0$ in $\Hs$ we get,
\be\no
\langle u,\phi\rangle_{\dot{H}^s}=\frac{\al}{2_s^*}\int_{\Rn}u^{\al-1}v^{\ba}\phi\,{\rm d}x+\prescript{}{(\dot{H}^s)'}\langle f,\phi\rangle_{\dot{H}^s}\geq 0.
\ee
This implies $0\leq u\in\Hs$ is a weak supersolution to $(-\De)^su=0$.
Therefore applying maximum principle \cite[Theorem 1.2(2)]{DPQ}, with $c=0$ and $p=2$ there, we obtain $u>0$ in $\Rn$. Similarly we can show that $v>0$ in $\Rn$. Hence, if $(u,v)$ is a critical point of $J_{f,g}$ then $(u,v)$ is a solution of \eqref{MAT1}.}
\end{remark}

To prove, existence of two critical points for $J_{f,g}$, first we partition the space $\Hs\times\Hs$ into three disjoint sets via the function $\Psi:\Hs\times\Hs\to\R$ defined by
\be
\Psi(u,v):=\|(u,v)\|_{\phs}^2-(2_s^*-1)\int_{\Rn}|u|^{\al}|v|^{\ba}\,{\rm d}x.\no
\ee
Set
$$\Omega_1:=\{(u,v)\in\Hs\times\Hs: (u,v)=(0,0) \,\,\text{or}\,\, \Psi(u,v)>0\},$$
 $$\Omega_2:=\{(u,v)\in\Hs\times\Hs: \Psi(u,v)<0\},$$
$$\Omega:=\{(u,v)\in\Hs\times\Hs\setminus\{(0,0)\}: \Psi(u,v)=0\}.$$
Put
\be\label{S34}
c_0:=\inf_{\Omega_1}J_{f,g}(u,v),\quad c_1:= \inf_{\Omega}J_{f,g}(u,v).
\ee
\begin{remark}\label{SR33} {\rm
Note that	for all $\la>0$ and $(u,v)\in\Hs\times\Hs$
	$$\Psi(\la u,\la v)=\la^2\|(u,v)\|_{\phs}^2-\la^{2_s^*}(2_s^*-1)\int_{\Rn}|u|^{\al}|v|^{\ba}\;{\rm d}x .$$
	Moreover, $\Psi(0,0)=0$ and $\la\mapsto\Psi(\la u,\la v)$ is a strictly concave function in $\R^+$. Thus for any $(u,v)\in\Hs\times\Hs$ with $\|(u,v)\|_{\phs}=1$, there exists a unique $\la$ ($\la$ depends on $(u,v)$) such that $(\la u,\la v)\in\Omega$. Moreover as
	$$\Psi(\la u,\la v)=(\la^2-\la^{2_s^*})\|(u,v)\|_{\phs}^2\;\, \text{for all }(u,v)\in\Omega,$$
	$(\la u,\la v)\in \Omega_1$ for all $\la\in(0,1)$ and $(\la u,\la v)\in\Omega_2$ for all $\la>1.$}
\end{remark}

\begin{lemma}\label{SL32}
	Assume $C_0$ is defined as in Theorem \ref{th:ex-f} and $c_0$ and $c_1$ are defined as in \eqref{S34}. Further, if
	\be\lab{J1.3}
	\inf_{(u,v)\in\dot{H}^s\times\dot{H}^s,\, \int_{\Rn}|u|^{\al}|v|^{\ba}\;{\rm d}x=1}\bigg\{C_0\|(u,v)\|_{\phs}^\frac{N+2s}{2s}
	-\prescript{}{(\dot{H}^s)'}{\langle}f,u{\rangle}_{H^s}-\prescript{}{(\dot{H}^s)'}{\langle}g,v{\rangle}_{H^s}\bigg\}>0,
	\ee
	then $c_0<c_1$.
\end{lemma}
	\begin{proof}
\underline{Step 1}: First we assert that, there exists $\delta>0$ such that
		$$\frac{d}{dt}I_{f,g}(tu,tv)\bigg{|}_{t=1}\geq \de \quad\forall\,\,  (u,v)\in \Omega.$$
	Doing a straight forward computation, it is easy to see that for any $(u,v)\in \Omega$

		\begin{align}\lab{30-7-4}
			\frac{d}{dt}\tilde I_{f,g}(tu,tv)\bigg{|}_{t=1}&=\frac{4s}{N+2s}\|(u,v)\|_{\phs}^2-\prescript{}{(\dot{H}^s)'}{\langle}f,u{\rangle}_{\dot{H}^s}-\prescript{}{(\dot{H}^s)'}{\langle}g,v{\rangle}_{\dot{H}^s}\no\\
&=C_0\frac{\|(u,v)\|^\frac{N+2s}{2s}_{\phs}}{\bigg(\displaystyle\int_{\Rn}|u|^{\al}|v|^{\ba}\;{\rm d}x\bigg)^\frac{N-2s}{4s}}-\prescript{}{(\dot{H}^s)'}{\langle}f,u{\rangle}_{\dot{H}^s}-\prescript{}{(\dot{H}^s)'}{\langle}g,v{\rangle}_{\dot{H}^s}
		\end{align}
		Further, $\eqref{J1.3}$ implies there  exists $d>0$ such that
		
		\be\lab{30-7-3}
	\inf_{(u,v)\in\Hs\times\Hs,\, \int_{\Rn}|u|^{\al}|v|^{\ba}\;{\rm d}x=1}\bigg\{C_0\|(u,v)\|_{\phs}^\frac{N+2s}{2s}
	-\prescript{}{(\dot{H}^s)'}{\langle}f,u{\rangle}_{H^s}-\prescript{}{(\dot{H}^s)'}{\langle}g,v{\rangle}_{H^s}\bigg\}\geq d.\ee
		Now, 
		\Bea
		\eqref{30-7-3}&\Longleftrightarrow& C_0\frac{\|(u,v)\|^\frac{N+2s}{2s}_{\phs}}{(\int_{\Rn}|u|^{\al}|v|^{\ba}\;{\rm d}x)^\frac{N-2s}{4s}}-\prescript{}{(\dot{H}^s)'}{\langle}f,u{\rangle}_{\dot{H}^s}-\prescript{}{(\dot{H}^s)'}{\langle}g,v{\rangle}_{\dot{H}^s} \geq d, \, \text{with } \int_{\Rn}|u|^{\al}|v|^{\ba}\;{\rm d}x=1.\\
		&\Longleftrightarrow& C_0\frac{\|(u,v)\|^\frac{N+2s}{2s}_{\phs}}{(\int_{\Rn}|u|^{\al}|v|^{\ba}\;{\rm d}x)^\frac{N-2s}{4s}}-\prescript{}{(\dot{H}^s)'}{\langle}f,u{\rangle}_{\dot{H}^s}-\prescript{}{(\dot{H}^s)'}{\langle}g,v{\rangle}_{\dot{H}^s} \geq d\left(\int_{\Rn}|u|^{\al}|v|^{\ba}\;{\rm d}x\right)^{1/2_s^*},\no\\
		&\;&\quad \qquad \qquad\qquad \qquad \qquad \quad \qquad \qquad \qquad \forall\,\, (u,v)\in\Hs\times\Hs\setminus\{(0,0)\}.
		\Eea
Observe that $\int_{\Rn}|u|^{\al}|v|^{\ba}\,{\rm d}x$ is bounded away from $0$ for all $(u,v)\in\Omega.$ Therefore, plugging back the above estimate into \eqref{30-7-4} proves Step 1.
		
		\vspace{2mm}
		
		\underline{Step 2:} Let $\{(u_n,v_n)\}$ be a minimizing sequence for $J_{f,g}$ on $\Omega$, i.e., $J_{f,g}(u_n,v_n)\to c_1$ and $\|(u_n,v_n)\|_{\phs}^2=(2^*_s-1)\int_{\Rn}|u_n|^{\al}|v_n|^{\ba}\;{\rm d}x$. Therefore, for large $n$
\begin{align*}c_1+o(1)\geq J_{f,g}(u_n,v_n)\geq I_{f,g}(u_n,v_n)&\geq \bigg(\frac{1}{2}-\frac{1}{2^*_s(2^*_s-1)}\bigg)\|(u_n,v_n)\|^2_{\phs}\\
		&\qquad- (\|f\|_{(\dot{H}^s)'}+\|g\|_{(\dot{H})'})\|(u_n, v_n)\|_{\phs}.\end{align*}
		This implies that $\{I_{f,g}(u_n,v_n)\}$ is a bounded sequence and $\{\|(u_n,v_n)\|_{\phs}\}$ and
		$\left\{\int_{\Rn}|u_n|^{\al}|v_n|^{\ba}\;{\rm d}x\right\}$ are bounded.
		
		{\bf Claim}: $c_0<0$.
		
		Observe that to prove the claim, it is sufficient to show that there exists $(u,v)\in \Omega_1$ such that $J_{f,g}(u,v)<0$. Using Remark \ref{SR33}, we can choose $(u,v)\in \Omega$ such that $\prescript{}{(\dot{H}^s)'}{\langle}f,u{\rangle}_{\dot{H}^s}+\prescript{}{(\dot{H}^s)'}{\langle}g,v{\rangle}_{\dot{H}^s}>0$. Therefore,
		$$J_{f,g}(tu,tv)=t^2\bigg[\frac{2^*_s-1}{2}-\frac{t^{2^*_s-2}}{2^*_s}\bigg]\int_{\Rn}|u|^{\al}|v|^{\ba}\;{\rm d}x-t\prescript{}{(\dot{H}^s)'}{\langle}f,u{\rangle}_{\dot{H}^s}-t\prescript{}{(\dot{H}^s)'}{\langle}g,v{\rangle}_{\dot{H}^s} <0$$
		for $t<<1$. Moreover, $(tu,tv)\in \Omega_1$
		by Remark \ref{SR33}. Hence the claim follows.
		
		Due to the above claim, $J_{f,g}(u_n,v_n)< 0$ for large $n$. Consequently, for large $n$
		$$0>J_{f,g}(u_n,v_n)\geq \bigg(\frac{1}{2}-\frac{1}{2^*_s(2^*_s-1)}\bigg)\|(u_n,v_n)\|^2_{\phs}-\prescript{}{(\dot{H}^s)'}{\langle}f,u_n{\rangle}_{\dot{H}^s}-\prescript{}{(\dot{H}^s)'}{\langle}g,v_n{\rangle}_{\dot{H}^s}.$$
		This in turn implies $\prescript{}{(\dot{H}^s)'}{\langle}f,u_n{\rangle}_{\dot{H}^s}+\prescript{}{(\dot{H}^s)'}{\langle}g,v_n{\rangle}_{\dot{H}^s}> 0$ for all large $n$.
		Consequently, $\frac{d}{dt}I_{f,g}(tu_n,tv_n)<0$ for $t>0$ small enough. Thus, by Step 1, there exists $t_n\in (0,1)$ such that $\frac{d}{dt}I_{f,g}(t_n u_n,t_n v_n)=0$. Since for all $(u,v)\in \Omega$, the function  $\frac{d}{dt}I_{f,g}(tu,tv)$ is strictly increasing in $t\in [0,1)$, we can conclude that
		$t_n$ is unique.\medskip
		
		\underline{Step 3}: In this step we show  that 	
		\be\label{J3}
		\liminf_{n\rightarrow\infty}\{I_{f,g}(u_n,v_n)-I_{f,g}(t_nu_n,t_nv_n)\}>0.
		\ee
Observe that $I_{f,g}(u_n,v_n)-I_{f,g}(t_nu_n,t_nv_n)=\displaystyle\int_{t_n}^{1}\frac{d}{dt}\{I_{f,g}(tu_n,tv_n)\} \, {\rm d}t$ and that for all $n\in\mathbb{N}$ there is $\xi_n>0$ such that
		$t_n\in(0,\;1-2\xi_n)$ and $\frac{d}{dt}I_{f,g}(tu_n)\geq \delta/2$ for $t\in[1-\xi_n,\;1]$.\\
		
		To establish \eqref{J3}, it is enough to show that $\xi_n>0$ can be chosen independent of $n\in\mathbb{N}$. This is possible as $\frac{d}{dt}I_{f,g}(tu_n,tv_n)|_{t=1}\geq \frac{\delta}{2}$ for $t\in[1-\xi_n, 1]$ and $\{(u_n,v_n)\}$ is bounded in $\Hs\times\Hs,$ so that for all $n\in\mathbb{N}$ and $t\in [0,1]$ \\
		\bea
		\bigg|\frac{d^2}{dt^2}I_{f,g}(tu_n,tv_n)\bigg| &=& \bigg|\|(u_n,v_n)\|_{\phs}^2 - (2^*_s-1)t^{2^*_s-2}\int_{\Rn}|u_n|^{\al}|v_n|^{\ba}\;{\rm d}x\bigg|\no\\
		&=& |1-t^{2^*_s-2}|\|(u_n,v_n)\|_{\phs}^2\leq C,\no
		\eea
		for all $n\geq 1$ and $t\in[0,\;1]$.
		
		\vspace{2mm}	
		
		\underline{Step 4:} From the definition of $J_{f,g}$ and $I_{f,g}$, it immediately follows that $\frac{d}{dt}J_{f,g}(tu,tv)\geq \frac{d}{dt}I_{f,g}(tu,tv)$ for all $(u,v)\in\Hs\times\Hs$ and for all $t>0$. Hence,
		\begin{align*}J_{f,g}(u_n,v_n)-J_{f,g}(t_nu_n,t_nv_n)= \int_{t_n}^{1}\frac{d}{dt}(J_{f,g}(tu_n,tv_n))\; {\rm d}t &\geq \int_{t_n}^{1}\frac{d}{dt}I_{f,g}(tu_n,tv_n)\; {\rm d}t \\
		&= I_{f,g}(u_n,v_n)-I_{f,g}(t_nu_n,t_nv_n).\end{align*}
		Since $\{(u_n,v_n)\}\subset \Omega$  is a minimizing sequence for $J_{f,g}$ on $\Omega$, and $(t_nu_n,t_nv_n)\in \Omega_1,$ we conclude using \eqref{J3} that\\
		$$c_0 = \inf_{(u,v)\in \Omega_1}J_{f,g}(u,v)< \inf_{(u,v)\in \Omega}J_{f,g}(u,v)= c_1.$$
	\end{proof}

	\begin{proposition}\lab{p:30-7-1}
		Assume that \eqref{J1.3} holds. Then $J_{f,g}$ has a critical point $(u_0,v_0)\in \Omega_1$, with $J_{f,g}(u_0,v_0)=c_0$. In particular, $(u_0,v_0)$ is a positive weak solution to \eqref{MAT1}.
	\end{proposition}
	
	\begin{proof} We decompose the proof into few steps.
		
		\vspace{2mm}
		
		\underline{Step 1}: $c_0>-\infty$.
		
		Since $J_{f,g}(u,v)\geq I_{f,g}(u,v)$, it is enough to show that $I_{f,g}$ is bounded from below. From the definition of $\Omega_1$, it immediately follows that for all $(u,v)\in {\Omega}_1$,
		\be\lab{31-7-3}I_{f,g}(u,v)\geq \left[\frac{1}{2}-\frac{1}{2^*_s(2^*_s-1)}\right]\|(u,v)\|_{\phs}^2- (\|f\|_{(\dot{H}^s)'}+\|g\|_{(\dot{H})'})\|(u_n, v_n)\|_{\phs}.\ee
		As RHS is quadratic function in $\|(u,v)\|_{\phs}$, $I_{f,g}$ is bounded from below. Hence Step 1 follows.
		
		\vspace{2mm}
		
		\underline{Step 2}: In this step we show that there exists a bounded  nonnegative $(PS)$ sequence
		$\{(u_n,v_n)\} \subset \Omega_1$ for $J_{f,g}$ at level $c_0$.
		
		Let $\{(u_n,v_n)\}\subset \bar{\Omega}_1$ such that $J_{f,g}(u_n,v_n)\to c_0$.
		Since Lemma \ref{SL32} implies that $c_0<c_1$,  without restriction we can assume $\{(u_n,v_n)\}\subset \Omega_1$. Further, using Ekeland's variational principle from $\{(u_n,v_n)\} $, we can extract a $(PS)$ sequence in $\Omega_1$ for $J_{f,g}$ at level $c_0$. We again call it by $\{(u_n,v_n)\}$. Moreover, as $J_{f,g}(u,v)\geq I_{f,g}(u,v)$, from \eqref{31-7-3} it follows that $\{(u_n,v_n)\} $ is a bounded sequence. Therefore,  up to a subsequence $(u_n,v_n)\rightharpoonup (u_0,v_0)$ in $\Hs\times\Hs$ and $(u_n,v_n)\to (u_0,v_0)$ a.e. in $\Rn$. In particular, $ (u_n)_+\to (u_0)_+,\;(v_n)_+\to (v_0)_+$ and $ (u_n)_-\to (u_0)_-, \;(v_n)_-\to (v_0)_-$ a.e. in $\Rn$. Moreover, as $f,\;g$ are nonnegative functionals, a straight forward computation yields
		\Bea
		o(1)&=&\prescript{}{\dot{H}^s)'}\langle J'_{f,g}(u_n,v_n), \left((u_n)_-,(v_n)_-\right)\rangle_{\phs}\\
		&=&\langle (u_n,v_n), \left((u_n)_-,(v_n)_-\right)\rangle_{\phs}-\prescript{}{(\dot{H}^s)'}\langle f, (u_n)_-\rangle_{\dot{H}^s}-\prescript{}{(\dot{H}^s)'}\langle g,(v_n)_-\rangle_{\dot{H}^s}\\
		&\leq&-\|\left((u_n)_-,(v_n)_-\right)\|_{\phs}^2.
		\Eea
		Therefore, $\big((u_n)_-,(v_n)_-\big)\To (0,0)$ in $\Hs\times \Hs$, which in turn implies up to a subsequence $(u_n)_-\to 0$ and $(v_n)_-\to 0$ a.e. in $\Rn$ and thus $(u_0)_-=0$  and $(v_0)_-=0$ a.e. in $\Rn$.  Consequently, without loss of generality, we can assume that $\{(u_n,v_n)\}$ is a nonnegative sequence. This completes the proof of Step~2.
		
		\vspace{2mm}
		
		\underline{Step 3:}  In this step we show that $(u_n,v_n)\to (u_0,v_0)$ in $\Hs\times \Hs$ and $(u_0,v_0)\in \Omega_1$.
		
		\noindent Applying Proposition \ref{PSP}, we get
		\be\label{J5}
		(u_n,v_n)-\bigg((u_0,v_0) +\sum_{j=1}^{m}(\tilde u_j,\tilde v_j)^{r_n^j, x_n^j}\bigg) \To (0,0) \;\text{ in }\Hs\times\Hs.
		\ee
		with $J_{f,g}'(u_0,v_0) =0$, $(\tilde u_j,\tilde v_j)$  is a nonnegative solution of \eqref{PSD2} ($(u_n, v_n)$ is $(PS)$ sequence for $J_{f,g}$ implies $(\tilde u_j,\tilde v_j)$  a solution of \eqref{S32} with $f=0=g$ and therefore by Remark~\ref{SR31}, $(\tilde u_j,\tilde v_j)$  is a nonnegative solution of \eqref{PSD2}), and  $\{x_n^j\}\subset \Rn$,  $\{r_n^j\}\subset\R^{+} $
are some appropriate sequences such that $r_n^j\overset{n\to\infty}\To 0$ and
		either $x_n^j\overset{n\to\infty}\To x^j$ or $|x_n^j|\overset{n\to\infty}\To\infty$. To prove Step~3, we need to show that $m=0$. Arguing  by contradiction,
		suppose that $j\neq 0$ in \eqref{J5}. Therefore,
		\begin{align}\lab{12-4-3}
		\Psi\bigg((\tilde u_j,\tilde v_j)^{r_n^j, x_n^j}\bigg)&=\|\left(\tilde u_j,\tilde v_j\right)\|^2_{\phs}-(2^*_s-1)\int_{\Rn}|\tilde u_j|^{\al}|\tilde{v}_j|^{\ba}\;{\rm d}x\no\\
&=(2-2_s^*)\|\left(\tilde u_j,\tilde v_j\right)\|_{\phs}^2<0.
		\end{align}			
		From Proposition \ref{PSP}, we also have
		$$c_0=\lim_{n\to\infty} J_{f,g}(u_n,v_n)=  J_{f,g}(u_0,v_0)+\sum_{j=1}^{m}J_{0,0}(\tilde u_j,\tilde v_j).$$
Since $(\tilde u_j,\tilde v_j)$ is a solution to \eqref{PSD2}, it is easy to see that $J_{0,0}(\tilde{u}_j,\tilde{v}_j) =\frac{s}{N}\|(\tilde u_j,\tilde v_j)\|_{\phs}^2$ and $S_{(\al,\ba)}\leq \|(\tilde u_j,\tilde v_j)\|_{\phs}^{\frac{4s}{N}}$, which in turn implies $J_{0,0}(\tilde u_j,\tilde v_j)= \frac{s}{N}S_{(\al,\ba)}^{\frac{N}{2s}}.$
		Consequently, $ J_{f,g}(u_0,v_0)<c_0$. Therefore, $(u_0,v_0)\not\in \Omega_1$ and
		\be\lab{12-4-4} \Psi(u_0,v_0)\leq 0.\ee 	
Next, we evaluate $\Psi\bigg((u_0,v_0) +\sum_{j=1}^{m}(\tilde u_j,\tilde v_j)^{r_n^j, x_n^j}\bigg)$. We observe that $(u_n,v_n)\in \Omega_1$ implies $\Psi(u_n,v_n)\geq 0$. Combining this with the uniform continuity of $\Psi$ and \eqref{J5} yields
		\be\label{J8}
		0\leq \liminf_{n\rightarrow\infty}\Psi(u_n,v_n)=\liminf_{n\rightarrow\infty} \Psi\bigg((u_0,v_0) +\sum_{j=1}^{m}(\tilde u_j,\tilde v_j)^{r_n^j, x_n^j}\bigg).
		\ee
		Note that from Step~2, we already have $u_0,\, v_0\geq 0$ and $(\tilde{u}_j, \tilde{v}_j)$ is nonnegative for all $j$ (see the paragraph after \eqref{J5}). Therefore, as $\al, \ba>1$
		\begin{align}\lab{12-4-1}
		\Psi\bigg((u_0,v_0) +\sum_{j=1}^{m}(\tilde u_j,\tilde v_j)^{r_n^j, x_n^j}\bigg)&=\|(u_0,v_0)\|_{\phs}^2+\Bigg{\|}\sum_{j=1}^{m}(\tilde u_j,\tilde v_j)^{r_n^j, x_n^j}\Bigg{\|}^2_{\phs}\no\\
		&\qquad+2\Big\langle (u_0,v_0),  \sum_{j=1}^{m} (\tilde u_j,\tilde v_j)^{r_n^j, x_n^j}\Big\rangle_{\phs}\no\\
		&\qquad- (2^*_s-1)\int_{\Rn}\bigg|u_0+\sum_{j=1}^{m}\tilde u_j^{r_n^j, x_n^j}   \bigg|^{\al}\bigg|v_0+\sum_{j=1}^{m}\tilde u_j^{r_n^j, x_n^j}\bigg|^{\ba}\;{\rm d}x\no\\
		&\leq\|(u_0,v_0)\|_{\phs}^2+ \sum_{j=1}^{m}\Bigg{\|}(\tilde u_j,\tilde v_j)^{r_n^j, x_n^j}\Bigg{\|}^2_{\phs}\no\\
		&\qquad+2\sum_{j=1}^{m}\sum_{i=1}^{m} \Big\langle (\tilde u_j,\tilde v_j)^{r_n^j, x_n^j}, (\tilde u_i,\tilde v_i)^{r_n^i, x_n^i}\Big\rangle_{\phs}\no\\
		&\qquad+2\Big\langle (u_0,v_0),  \sum_{j=1}^{m} (\tilde u_j,\tilde v_j)^{r_n^j, x_n^j}\Big\rangle_{\phs}\no\\
		&\qquad-(2^*_s-1)\bigg(\int_{\Rn}|u_0|^\al|v_0|^\ba{\rm d}x+\sum_{j=1}^{m}\int_{\Rn}|\tilde{u}_j^{r_n^j, x_n^j}|^\al |\tilde{v}_j^{r_n^j,x_n^j}|^\ba {\rm d}x \bigg)\no\\
&=\Psi(u_0,v_0)+\sum_{j=1}^{m}\Psi\big((\tilde u_j,\tilde v_j)^{r_n^j, x_n^j}\big)\no\\
&\qquad+ \text{the above inner products.}
		\end{align}
We now prove that all the inner products in the RHS of \eqref{12-4-1} approaches $0$ as $n\to\infty$. As $r_n^j\overset{n\to\infty}\To 0$, it follows that $u_j^{r_n^j, x_n^j}\deb 0$ and $v_j^{r_n^j, x_n^j}\deb 0$ in $\Hs$ as $n\to\infty$ (see \cite[Lemma~3]{PS}). Therefore, $\langle u_0, u_j^{r_n^j, x_n^j}\rangle_{\dot{H}^s} \overset{n\to\infty}\To 0$ and $\langle v_0, v_j^{r_n^j, x_n^j}\rangle_{\dot{H}^s}\overset{n\to\infty}\To 0$ for all $j=1,\cdots, m$. Hence,
$$2\Big\langle (u_0,v_0),  \sum_{j=1}^{m} (\tilde u_j,\tilde v_j)^{r_n^j, x_n^j}\Big\rangle_{\phs}=o(1) \quad\text{as}\,\, n\to\infty.$$		
Next,
		\Bea
		&&\Big\langle (\tilde u_j,\tilde v_j)^{r_n^j, x_n^j}, (\tilde u_i,\tilde v_i)^{r_n^i, x_n^i}\Big\rangle_{\phs}\\
		&=&(r_n^i)^\frac{N-2s}{2}(r_n^j)^{-\frac{N-2s}{2}} \iint_{\R^{2N}}
	\frac{\big(\tilde u_j(\frac{x-x_n^j}{r_n^j})-\tilde u_j(\frac{y-x_n^j}{r_n^j})\big)\big(\tilde u_i(\frac{x-x_n^i}{r_n^i})-\tilde u_i(\frac{x-x_n^i}{r_n^i})\big)}{|x-y|^{N+2s}}{\rm d}x{\rm d}y\\
	&&+(r_n^i)^\frac{N-2s}{2}(r_n^j)^{-\frac{N-2s}{2}} \iint_{\R^{2N}}
	\frac{\big(\tilde v_j(\frac{x-x_n^j}{r_n^j})-\tilde v_j(\frac{y-x_n^j}{r_n^j})\big)\big(\tilde v_i(\frac{x-x_n^i}{r_n^i})-\tilde v_i(\frac{x-x_n^i}{r_n^i})\big)}{|x-y|^{N+2s}} {\rm d}x{\rm d}y\\	
&=&(r_n^i)^\frac{N-2s}{2}(r_n^j)^{-\frac{N-2s}{2}} \iint_{\R^{2N}}\frac{\big(\tilde u_i(x)-\tilde u_i(y)\big)\big(\tilde u_j(\frac{r_n^ix+x_n^i-x_n^j}{r_n^j})-\tilde u_j(\frac{r_n^iy+x_n^i-x_n^j}{r_n^j})\big)}{|x-y|^{N+2s}}{\rm d}x{\rm d}y\\
		&&+(r_n^i)^\frac{N-2s}{2}(r_n^j)^{-\frac{N-2s}{2}} \iint_{\R^{2N}}\frac{\big(\tilde v_i(x)-\tilde v_i(y)\big)\big(\tilde v_j(\frac{r_n^ix+x_n^i-x_n^j}{r_n^j})-\tilde v_j(\frac{r_n^iy+x_n^i-x_n^j}{r_n^j})\big)}{|x-y|^{N+2s}}{\rm d}x{\rm d}y\\
		&=&\langle \tilde u_i, \tilde u_j^n\rangle_{\dot{H}^s}+
		\langle\tilde v_i, \tilde v_j^n \rangle_{\dot{H}^s},
		\Eea
		where $\tilde u_j^n(x):=\left(\frac{r_n^i}{r_n^j}\right)^\frac{N-2s}{2}\tilde u_j\big(\frac{r_n^i}{r_n^j}x+\frac{x_n^i-x_n^j}{r_n^j}\big)$,
		and $\tilde v_j^n(x):=\left(\frac{r_n^i}{r_n^j}\right)^\frac{N-2s}{2}\tilde v_j\big(\frac{r_n^i}{r_n^j}x+\frac{x_n^i-x_n^j}{r_n^j}\big).$ Further, we observe that  using the following
		$$\bigg|\log\left(\frac{r^i_n}{r_n^j}\right)\bigg|+\bigg|\frac{x_n^i-x_n^j}{r_n^i}\bigg|\To\infty$$
		from Proposition \ref{PSP}, it is easy to see that $\tilde u_i^n\deb 0$  and $\tilde v_i^n\deb 0$ in $\Hs\times\Hs$
		 as $n\to \infty$ for each fixed $i$ and $j$ (see \cite[Lemma~3]{PS}).  Hence
		 $$\Big\langle (\tilde u_j,\tilde v_j)^{r_n^j, x_n^j}, (\tilde u_i,\tilde v_i)^{r_n^i, x_n^i}\Big\rangle_{\phs}=o(1).$$
Substituting this back into \eqref{12-4-1} and using \eqref{12-4-3} and \eqref{12-4-4} gives a contradiction to \eqref{J8}. Therefore, $m=0$ in \eqref{J5}. Hence, $(u_n,v_n)\to (u_0,v_0)$ in $\Hs\times\Hs$. Consequently, $\Psi(u_n,v_n)\to \Psi(u_0,v_0)$, which in turn implies $(u_0,v_0)\in \bar \Omega_1$. But, since $c_0<c_1$, we conclude $(u_0,v_0)\in \Omega_1$. Thus Step 3 follows.

		\vspace{2mm}
		
		{\bf Step 4:} From the previous steps we see that $J_{f,g}(u_0,v_0)=c_0$ and
		$J_{f,g}'(u_0,v_0)=0$. Therefore, $(u_0,v_0)$ is a weak solution to \eqref{S32}. Combining this with Remark \ref{SR31}, we end the proof of the proposition.
	\end{proof}

	\begin{proposition}\lab{p:31-7-2}
		 Assume that \eqref{J1.3} holds. Then, $J_{f,g}$ has a second critical point $(u_1,v_1)\neq (u_0,v_0)$, where $(u_0,v_0)$ is the positive solution to \eqref{MAT1} obtained in Proposition \ref{p:30-7-1}.
In particular, $(u_1,v_1)$ is a second positive solution to \eqref{MAT1}.
	\end{proposition}

\begin{proof}
		Let $(u_0,v_0)$ be the critical point obtained in Proposition \ref{p:30-7-1} and $(Bw, Cw)$ (with $C=B\sqrt{\frac{\ba}{\al}}$) be a positive ground state solution of \eqref{PSD2} described as in Theorem~\ref{th:uni}. A standard computation yields that $I_{0,0}(Bw, Cw)=\displaystyle\frac{s}{N}S_{\al,\ba}^\frac{N}{2s}$. For $t>0$, define
		$$  w_t(x):=w\big(\frac{x}{t}\big), \quad \tilde u_t(x):=Bw_t(x), \quad \tilde v_t(x):=Cw_t(x).$$
		
\noindent{\bf Claim 1:} $(u_0+\tilde u_t,v_0+\tilde v_t)\in \Omega_2$ for $t>0$ large enough.
		
		Indeed, as $(u_0,v_0)$ and $(\tilde u_t, \tilde v_t)$ are positive and $\al, \ba>1$, using Young's inequality with $\eps>0$, we have
		\Bea
		\Psi(u_0+\tilde u_t,v_0+\tilde v_t)&=&\|(u_0+\tilde u_t)\|^2_{\dot{H}^s}+
		\|(v_0+\tilde v_t)\|^2_{\dot{H}^s}-(2^*_s-1)\int_{\Rn}|u_0+\tilde u_t|^\al|v_0+\tilde v_t|^\ba{\rm d}x\\
		&\leq&\|(u_0,v_0)\|^2_{\phs}+\|(\tilde u_t,\tilde v_t)\|^2_{\phs}+2\langle u_0, u_t\rangle_{\dot{H}^s}+2\langle v_0, v_t\rangle_{\dot{H}^s}\\
				&&\quad-(2^*_s-1)\left(\int_{\Rn}|u_0|^{\al}|v_0|^{\ba}\;{\rm d}x+\int_{\Rn}|\tilde u_t|^{\al}|\tilde v_t|^{\ba}\;{\rm d}x\right)\\
		&\leq&(1+\eps)\|(\tilde u_t,\tilde v_t)\|^2_{\phs}+(1+C_\eps)\|(u_0,v_0)\|^2_{\phs}\\
		&&-(2^*_s-1)\int_{\Rn}|u_0|^{\al}|v_0|^{\ba}{\rm d}x-(2^*_s-1)t^NB^{\al}C^{\ba}\int_{\Rn}|w|^{2_s^*}{\rm d}x\\
		&\leq&\bigg((1+\eps)(B^2+C^2)t^{N-2s}-(2^*_s-1)t^NB^{\al}C^{\ba}\bigg)\|w\|_{\dot{H}^s}^2\\
		&&+(1+C_\eps)\|(u_0,v_0)\|^2_{\phs}-(2^*_s-1)\int_{\Rn}|u_0|^{\al}|v_0|^{\ba}{\rm d}x\\
		&<& 0 \quad\mbox{for } t>0 \mbox{  large enough}.
		\Eea
		 Hence the claim follows.
		
		\vspace{2mm}
		
		\noindent{\bf Claim 2:} $J_{f,g}\bigg(u_0+\tilde u_t,v_+\tilde v_t\bigg)<J_{f,g}(u_0,v_0)+J_{0,0}(\tilde u_t,\tilde v_t)\,\;\forall\,\, t>0 $.
		
		Indeed, since $u_0,\, v_0, \, w_t,\, B>0$, taking $(\tilde u_t,\tilde v_t)$ as the test function for \eqref{S32} yields
		 \begin{align*} {\Big\langle} (u_0,v_0),\;(\tilde u_t,\tilde v_t){\Big\rangle}_{\phs} &=\frac{\al}{2_s^*}\int_{\Rn}u_0^{\al-1}v_0^{\ba}\tilde u_t\, {\rm d}x+\frac{\ba}{2_s^*}\int_{\Rn}u_0^{\al}v_0^{\ba-1}\tilde v_t\;{\rm d}x\\
		 &\qquad\qquad+\prescript{}{(\dot{H}^s)'}{\langle}f,\tilde u_t{\rangle}_{\dot{H}^s}+\prescript{}{(\dot{H}^s)'}\langle g,\tilde v_t\rangle_{\dot{H}^s}.\end{align*}
		Consequently, using the above expression, we obtain
		\begin{align*}
		J_{f,g}\big(u_0+\tilde u_t,v_0+\tilde v_t\big)
		&= \frac{1}{2}\|(u_0,v_0)\|_{\phs}^2+\frac{1}{2}\|(\tilde u_t, \tilde v_t)\|_{\phs}^2+ {\Big\langle} (u_0,v_0),\;(\tilde u_t,\tilde v_t){\Big\rangle}_{\phs} \\
		&\quad-\frac{1}{2^*_s}\int_{\Rn}(u_0+\tilde u_t)^{\al}(v_0+\tilde v_t)^{\ba}\;{\rm d}x-\prescript{}{(\dot{H}^s)'}{\langle}f,u_0+\tilde u_t{\rangle}_{\dot{H}^s}-\prescript{}{(\dot{H}^s)'}{\langle}g,v_0+\tilde v_t{\rangle}_{\dot{H}^s}\\
		&=J_{f,g}(u_0,v_0)+J_{0,0}(\tilde u_t,\tilde v_t) +\frac{1}{2_s^*}\int_{\Rn}u_0^{\al}v_0^{\ba}\;{\rm d}x+\frac{1}{2_s^*}\int_{\Rn}|\tilde u_t|^{\al}|\tilde v_t|^{\ba}{\rm d}x\\
		&\quad+\frac{\al}{2^*_s}\int_{\Rn}u_0^{\al-1}v_0^{\ba}\tilde u_t\;{\rm d}x
		+\frac{\ba}{2^*_s}\int_{\Rn}u_0^{\al}v_0^{\ba-1}\tilde v_t\;{\rm d}x \\
		&\quad-\frac{1}{2^*_s}\int_{\Rn}(u_0+\tilde u_t)^{\al}(v_0+\tilde v_t)^{\ba}\;{\rm d}x\\
		&\leq J_{f,g}(u_0,v_0)+J_{0,0}(\tilde u_t,\tilde v_t)  +\frac{1}{2^*_s}\int_{\Rn}\bigg[u_0^{\al}v_0^{\ba}+\tilde u_t^{\al}\tilde v_t^{\ba}\\
		&\qquad\qquad+\al u_0^{\al-1}v_0^{\ba}\tilde u_t+\ba u_0^{\al}v_0^{\ba-1}\tilde v_t-(u_0+\tilde u_t)^{\al}(v_0+\tilde v_t)^{\ba}\bigg]{\rm d}x\\
		&< J_{f,g}(u_0,v_0)+J_{0,0}(\tilde u_t,\tilde v_t) .
		\end{align*}
		Hence the Claim follows.
		
		\vspace{2mm}
		
		Using the definition of $\tilde u_t$ and  $\tilde v_t$, it immediately follows
		\be\lab{1-8-2}
		\lim_{t\to\infty}J_{0,0}\big(\tilde u_t,\tilde v_t\big)=-\infty,
				\ee
and
		$$\sup_{t>0} J_{0,0}\big(\tilde u_t,\tilde v_t\big)=J_{0,0}(\tilde u_{t'},\tilde v_{t'}), \quad\text{where}\quad t'=\bigg(\frac{B^2+C^2}{B^\al C^\ba}\bigg)^\frac{1}{2s}.$$
Therefore,	 doing a straight forward computation and using Lemma~\ref{l:S}, we get that
$$\sup_{t>0} J_{0,0}\big(\tilde u_t,\tilde v_t\big)=\frac{s}{N}\frac{(B^2+C^2)^\frac{N}{2s}}{(B^\al C^\ba)^\frac{N-2s}{2s}}S^\frac{N}{2s}=\frac{s}{N}S_{\al,\ba}^\frac{N}{2s}.$$
Combining this with Claim 2 and \eqref{1-8-2} yields
		\be\begin{split}\lab{1-8-6}
			J_{f,g}(u_0+\tilde u_t,v_0+\tilde v_t)<J_{f,g}(u_0,v_0)+\frac{s}{N}S_{\al,\ba}^{\frac{N}{2s}}  \quad\forall\, t>0 \\
			\mbox{ and }\qquad J_{f,g}(u_0+\tilde u_t,v_0+\tilde v_t)<J_{f,g}(u_0,v_0) \quad\mbox{for } t \mbox{ large enough}.
		\end{split}
		\ee
		Fix $t_0>0$ large enough such that \eqref{1-8-6} and Claim 1 are satisfied.
		
		Next, we set
		$$\eta:=\inf_{\ga\in\Ga}\max_{r\in[0,1]} J_{f,g}\big(\ga(r)\big),$$
		where $$\Ga:=\bigg\{\ga\in C\big([0,1],\, \Hs\times\Hs\big) :  \ga(0)=(u_0,v_0),\quad\ga(1)= (u_0+\tilde u_{t_0}, v_0+\tilde v_{t_0})\bigg\}.$$
As $(u_0,v_0)\in \Omega_1$ and $(u_0+\tilde u_{t_0},v_0+\tilde v_{t_0})\in \Omega_2$, for every $\ga\in \Ga$, there exists $r_\ga\in(0,1)$
		such that $\ga(r_\ga)\in \Omega$. Therefore,
		$$\max_{r\in[0,1]} J_{f,g}(\ga(r))\geq J_{f,g}\big(\ga(r_\ga)\big)\geq \inf_{\Omega}J_{f,g}(u,v)=c_1.$$
		Thus, $\eta\geq c_1>c_0=J_{f,g}(u_0,v_0)$.
		Here in the last inequality we have used Lemma \ref{SL32}.
		
		\vspace{2mm}

		\noindent{\bf Claim 3:} $\displaystyle J_{f,g}(u_0,v_0)<  \eta<J_{f,g}(u_0,v_0)+\frac{s}{N}S_{\al,\ba}^{\frac{N}{2s}}$.
		
Since $\lim_{t\to 0}\|w_t\|_{\Hs}=0$, we also have $\lim_{t\to 0}\|(\tilde u_t, \tilde v_t)\|_{\phs}=0$. Thus, if we define
		$\tilde \ga(r):=(u_0, v_0)+ (\tilde u_{rt_0}, \tilde u_{rt_0})$, then
		$\lim_{r\to 0}\|\tilde \ga(r)-(u_0, v_0)\|_{\phs}=0$. Consequently, $\tilde \ga\in \Ga$. Therefore, using \eqref{1-8-6}, we obtain

		$$\eta\leq \max_{r\in[0,1]}J_{f,g}(\tilde \ga(r))=\max_{r\in[0,1]}J_{f,g}\left(u_0+\tilde u_{rt_0}, v_0+\tilde v_{rt_0}\right)<J_{a,f}(u_0,v_0)+\frac{s}{N}S_{\al,\ba}^{\frac{N}{2s}}.$$
		Hence Claim 3 follows.

		Using Ekeland's variational principle, there exists a $(PS)$ sequence $\{(u_n,v_n)\}$ for $J_{f,g}$ at level $\eta$. Arguing as before we see that $\{(u_n,v_n)\}$ is a bounded sequence. Further,
since Claim 3 holds, from 	Proposition~\ref{PSP} we conclude that $(u_n,v_n)\to (u_1,v_1)$, for some $(u_1,v_1)\in \Hs\times\Hs$ such that $J_{f,g}'(u_1,v_1)=0$ and $J_{f,g}(u_1,v_1)=\eta$. On the other hand, as
		$J_{f,g}(u_0,v_0)<\eta$, we conclude $(u_0,v_0)\neq (u_1,v_1)$.
		
		$J_{f,g}'(u_1,v_1)=0\Longrightarrow (u_1,v_1)$ is a weak solution to \eqref{S32}. Combining this with Remark~\ref{SR31}, we complete the proof of the proposition.
	\end{proof}

	\begin{lemma}\lab{l:J1.3}
		Let $C_0$ be as defined in Theorem~\ref{th:ex-f}. If $\max\{\|f\|_{(\dot{H}^s)'},\;\|g\|_{(\dot{H}^s)'}\}<C_0S_{\al,\ba}^\frac{N}{4s}$, then \eqref{J1.3} holds.
	\end{lemma}
	\begin{proof}
	{\bf Assertion 1}: $$\frac{4s}{N+2s}\|(u,v)\|_{\phs} \geq C_0S_{\al,\ba}^{\frac{N}{4s}} \quad\forall\,\, (u, v)\in \Omega.$$
To see this, we fix $(u,v)\in\Omega$. Therefore, using the definition of $S_{\al,\ba}$ we have	
\bea
	\|(u,v)\|_{\phs}&\geq& S_{\al,\ba}^{1/2}\left(\int_{\Rn}|u|^{\al}|v|^{\ba}\;{\rm d}x\right)^{1/2_s^*}\no\\
	 \implies\|(u,v)\|_{\phs}&\geq& S_{\al,\ba}^{{1/2}} \frac{\|(u,v)\|_{\phs}^{2/2_s^*}}{\left(2^*_s-1\right)^{{1/2_s^*}}}.\no
	\eea
	From here, using the definition of $C_0$, the assertion follows.
	
Note that by the given hypothesis, there exists $\varepsilon>0$ such that
		\be
		\|f\|_{(\dot{H}^s)'}+\|g\|_{(\dot{H}^s)'}< C_0 S_{\al,\ba}^{\frac{N}{4s}}-\varepsilon.\no
		\ee
		Combining this with the above Assertion~1, for all $(u,v)\in\Omega$, it holds
		\begin{align*}
			\prescript{}{(\dot{H}^s)'}{\langle}f,u{\rangle}_{\dot{H}^s} + \prescript{}{(\dot{H}^s)'}{\langle}g,v{\rangle}_{\dot{H}^s} &\leq \left(\|f\|_{(\dot{H}^s)'}+\|g\|_{(\dot{H}^s)'}\right)\|(u,v)\|_{\phs}\\
			&< \left(C_0S_{\al,\ba}^{\tfrac{N}{4s}}-\varepsilon\right)\|(u,v)\|_{\phs}\\
			&\leq \frac{4s}{N+2s}\|(u,v)\|_{\phs}^2-\varepsilon\|(u,v)\|_{\phs}.
		\end{align*}
Consequently,
		$$	\inf_{(u,v)\in\Omega}\[\frac{4s}{N+2s}\|(u,v)\|_{\phs}^2-\prescript{}{(\dot{H}^s)'}{\langle}f,u{\rangle}_{\dot{H}^s} - \prescript{}{(\dot{H}^s)'}{\langle}g,v{\rangle}_{\dot{H}^s}\]> \varepsilon \inf_{(u,v)\in\Omega}\|(u,v)\|_{\phs}.
		$$
		Since  $\|(u,v)\|_{\phs}$ is bounded away from $0$ on $\Omega$, the above expression implies that
		\be\label{S311}
		\inf_{(u,v)\in\Omega}\[\frac{4s}{N+2s}\|(u,v)\|_{\phs}^2
		-\prescript{}{(\dot{H}^s)'}{\langle}f,u{\rangle}_{\dot{H}^s} - \prescript{}{(\dot{H}^s)'}{\langle}g,v{\rangle}_{\dot{H}^s}\]>0.
		\ee
		On the other hand,
		\begin{align}\label{S312}
			\eqref{J1.3}&\iff C_0\frac{\|(u,v)\|_{\phs}^{\frac{N+2s}{2s}}}
			{\bigg(\displaystyle\int_{\Rn}|u|^{\al}|v|^{\ba}\;{\rm d}x\bigg)^{\frac{N-2s}{4s}}}
			-\prescript{}{(\dot{H}^s)'}{\langle}f,u{\rangle}_{\dot{H}^s} - \prescript{}{(\dot{H}^s)'}{\langle}g,v{\rangle}_{\dot{H}^s}>0\no\\
			 &\qquad\qquad\qquad\qquad\qquad\qquad\qquad\qquad \mbox{ for }\int_{\Rn}|u|^{\al}|v|^{\ba}\;{\rm d}x=1\no\\
			&\iff  C_0\frac{\|(u,v)\|_{\phs}^{\frac{N+2s}{2s}}}
			{\bigg(\displaystyle\int_{\Rn}|u|^{\al}|v|^{\ba}\;{\rm d}x\bigg)^{\frac{N-2s}{4s}}}
			-\prescript{}{(\dot{H}^s)'}{\langle}f,u{\rangle}_{\dot{H}^s} - \prescript{}{(\dot{H}^s)'}{\langle}g,v{\rangle}_{\dot{H}^s}>0
			\quad\mbox{for }(u,v)\in\Omega\no\\
			&\iff \frac{4s}{N+2s}\|(u,v)\|_{\phs}^2	-\prescript{}{(\dot{H}^s)'}{\langle}f,u{\rangle}_{\dot{H}^s} - \prescript{}{(\dot{H}^s)'}{\langle}g,v{\rangle}_{\dot{H}^s}>0
			\quad\mbox{for }(u,v)\in\Omega.
		\end{align}
		Clearly, \eqref{S311} ensures that the RHS of \eqref{S312} holds. The lemma now follows.
\end{proof}

{\bf End of Proof of Theorem~\ref{th:ex-f}} Combining Propositions~\ref{p:30-7-1} and \ref{p:31-7-2} with Lemmas~\ref{l:J1.3} and \ref{SL32}, we obtain two positive solutions of Theorem~\ref{th:ex-f}. The last assertion of the theorem follows as proved in \cite[Theorem 1.1]{BCP}.

\appendix
\section{Product of Morrey spaces}
\setcounter{equation}{0}

First we recall the definition of the homogeneous Morrey spaces $L^{r,\ga}(\Rn)$, introduced by Morrey as a
refinement of the usual Lebesgue spaces. A measurable function $u:\Rn\to\R$
belongs to the Morrey space $L^{r,\ga}(\Rn)$, with $r\in[1,\infty)$ and $\ga\in[0,N]$ if and only if
\be\lab{12-13-1}
\|u\|_{L^{r,\ga}(\Rn)}^r:=\sup_{R>0,\, x\in\Rn} R^\ga\fint_{B(x,R)}|u|^r dy<\infty.
\ee
Note that if $\ga=N$ then $L^{r,N}(\Rn)$ coincides with the usual Lebesgue space $L^r(\Rn)$ for any $r\geq 1$ and similarly $L^{r,0}(\Rn)$ coincides with $L^\infty(\Rn)$. Also we observe that $L^{r,\ga}(\Rn)$ experiences same translation and dilation invariance as in $L^{2^*_s}(\Rn)$ and therefore of $\Hs$ if $\frac{\ga}{r}=\frac{N-2s}{2}$. Let $(u)^{x_0,r}$ be the function defined by \eqref{12-13-3}. By change of variable
formula, one can see that the following equality holds
$$\|(u)^{x_0,r}\|_{L^{r, \frac{N-2s}{2}r}}=\|u\|_{L^{r, \frac{N-2s}{2}r}},$$
for any $r\in[1, 2^*_s]$. We recall that there
exists a constant $C=C(N,s)$ such that
\be\lab{12-13-4}\|u\|_{L^{r, r(N-2s)/2}}\leq C\|u\|_{L^{2^*_s}} \quad\mbox{for all } u\in L^{2^*_s}(\Rn),\ee
see \cite[Theorem 1]{PS}(also see \cite[(A.2)]{BP}). For further discussion on Morrey spaces, we refer the reader to \cite{PS}. Next we define the product space 
$L^{2, N-2s}(\Rn)\times L^{2, N-2s}(\Rn)$ in the standard way with
$$\|(u,v)\|_{L^{2, N-2s}\times L^{2, N-2s}}:=\big(\|u\|^2_{L^{2, N-2s}}+\|v\|^2_{L^{2, N-2s}}\big)^\frac{1}{2}.$$
Therefore, using Sobolev inequality and \eqref{12-13-4}, it follows that
\be\lab{3-9-1}
\dot{H}^s\times\dot{H}^s\hookrightarrow L^{2^*_s}\times L^{2^*_s}\hookrightarrow L^{2,N-2s}\times L^{2,N-2s},
\ee
where the embedding are continuous.

\begin{lemma}\lab{l:12-13-1}
For any $0<s<N/2$  there exists a constant $C=C(N,s)$ such that, for any $2/2^*_s\leq \theta<1$ and for any $1\leq r<2^*_s$
$$\|(u,v)\|_{L^{2^*_s}\times L^{2^*_s}}\leq C\|(u,v)\|_{\dot{H}^s\times \dot{H}^s}^\theta\|(u,v)\|^{1-\theta}_{L^{2, (N-2s)}\times L^{2, (N-2s)}}$$
for all  $(u, v)\in\Hs\times\Hs$.
\end{lemma}

\begin{proof}
Using \cite[Theorem 1]{PS},
\Bea
\|(u,v)\|_{L^{2^*_s}\times L^{2^*_s}}&=&\big(\|u\|^2_{L^{2^*_s}}+\|v\|^2_{L^{2^*_s}}\big)^\frac{1}{2}\no\\
&\leq&\|u\|_{L^{2^*_s}}+\|v\|_{L^{2^*_s}}\no\\
&\leq&C\big[\|u\|_{\dot{H}^s}^\theta\|u\|^{1-\theta}_{L^{2, N-2s}}+
\|v\|_{\dot{H}^s}^\theta\|v\|^{1-\theta}_{L^{2, N-2s}}\big]\\
&\leq&C\big[\|(u,v)\|_{\dot{H}^s\times \dot{H}^s}^\theta\|u\|^{1-\theta}_{L^{2, N-2s}}+
\|(u,v)\|_{\dot{H}^s\times \dot{H}^s}^\theta\|v\|^{1-\theta}_{L^{2, N-2s}}\big]\\
&\leq&C\|(u,v)\|_{\dot{H}^s\times \dot{H}^s}^\theta\big[\|u\|^{1-\theta}_{L^{2, N-2s}}+\|v\|^{1-\theta}_{L^{2,N-2s}}\big]\\
&\leq&C\|(u,v)\|_{\dot{H}^s\times \dot{H}^s}^\theta\big[\|(u,v)\|^{1-\theta}_{L^{2, N-2s}\times L^{2, N-2s}}+\|(u,v)\|^{1-\theta}_{L^{2,N-2s}\times L^{2, N-2s}}\big]\\
&\leq&2C\|(u,v)\|_{\dot{H}^s\times \dot{H}^s}^\theta\|(u,v)\|^{1-\theta}_{L^{2, (N-2s)}\times L^{2, (N-2s)}}.
\Eea
\end{proof}

{\bf Acknowledgement}:  The research of M.~Bhakta is partially supported by the  SERB MATRICS grant MTR/2017/000168 and SERB WEA grant WEA/2020/000005.

S. Chakraborty is partially supported by NBHM grant 0203/11/2017/RD-II.

O. H. Miyagaki has received grants from CNPq/Brazil grant n$^{\underline{{\rm o}}}$ 307061/2018-3 and  FAPESP/ Brazil grant n$^{\underline{{\rm o}}}$ 2019/24901-3.

P. Pucci is a member of the {\em Gruppo Nazionale per
l'Analisi Ma\-te\-ma\-ti\-ca, la Probabilit\`a e le loro Applicazioni}
(GNAMPA) of the {\em Istituto Nazionale di Alta Matematica} (INdAM)
and partly supported by the INdAM -- GNAMPA Project
{\em Equazioni alle derivate parziali: problemi e mo\-del\-li} (Prot\_U-UFMBAZ-2020-000761). P. Pucci was also partly supported by of the {\em Fondo Ricerca di Base di Ateneo --
Eser\-ci\-zio 2017--2019} of the University of Perugia, named {\em PDEs and Nonlinear Analysis}.

\end{document}